\documentclass[11pt,reqno]{amsart}
\usepackage{enumerate}
\usepackage{amssymb,commath,amscd}
\usepackage{txfonts}
\usepackage{amsthm}
\usepackage{mathrsfs}
\usepackage{thmtools}
\usepackage{amsmath}
\usepackage{
}
\usepackage[colorlinks,linkcolor=black,anchorcolor=blue,citecolor=green]{hyperref}
\usepackage[all]{xy}

\declaretheorem[numberwithin=section]{theorem}
\declaretheorem[sibling=theorem]{lemma}

\declaretheorem[sibling=theorem]{proposition}

\declaretheorem[sibling=theorem]{problem}

\numberwithin{equation}{section}
\allowdisplaybreaks

\newcommand{\injrad}{\operatorname{inj.rad}}

\newcommand{\Ric}{\operatorname{Ric}}
\newcommand{\diam}{\operatorname{diam}}
\newcommand{\vol}{\operatorname{vol}}
\newcommand{\sndf}{{I\!I}}

\begin{document}
	\title[Canonical nilpotent structure with manifold orbit space] {Canonical nilpotent structure under bounded Ricci curvature and Reifenberg local covering geometry over regular limits}
	
	
	\author{Zuohai Jiang}
	\address[Zuohai Jiang]{Beijing International Center for Mathematical Research, Peking University, Beijing, China}
	\email{jiangzuohai08@pku.edu.cn}
	
	
	\author{Lingling Kong}
	\address[Lingling Kong]{School of Mathematical and statistical,  Northeast Normal University, Changchun, China}
	\curraddr{}
	\email{kongll111@nenu.edu.cn}
	\thanks{L.K. would like to thank Capital Normal University for
		a warm hospitality during his visit}
	\author{Shicheng Xu}
        \address[Shicheng Xu]{School of Mathematical Sciences, Capital Normal University, Beijing, China}
        \address[Shicheng Xu]{Academy for Multidisciplinary Studies, Capital Normal University, Beijing, China}
        \curraddr{}
        \email{shichengxu@gmail.com}
	
	
	\date{\today}
		\begin{abstract}
		It is known that a closed collapsed Riemannian $n$-manifold $(M,g)$ of bounded Ricci curvature and Reifenberg local covering geometry admits a nilpotent structure in the sense of Cheeger-Fukaya-Gromov with respect to a smoothed metric $g(t)$. We prove that a canonical nilpotent structure over a regular limit space that describes the collapsing of original metric $g$ can be defined and uniquely determined up to a conjugation, and prove that the nilpotent structures arising from nearby metrics $g_\epsilon$ with respect to $g_\epsilon$'s sectional curvature bound are equivalent to the canonical one.
	\end{abstract}
	\maketitle
	
	\setcounter{section}{-1}
	\section{Introduction}\label{section-0}
		A Riemannian manifold $(M,g)$ with a normalized curvature bound is called $\epsilon$-collapsed, if the volume of any unit ball $B_{1}(x)$ on $M$ is less than $\epsilon$.
	Collapsed Riemannian manifolds with bounded sectional curvature were extensively studied by \cite{Fukaya1987,Fukaya1988,Fukaya1989}, \cite{CheegerGromov1990,CheegerGromov1990II} and \cite{CFG1992}. A nilpotent structure, which totally characterizes ``collapsing implies symmetry'' phenomena under bounded sectional curvature, is constructed by  Cheeger-Fukaya-Gromov \cite{CFG1992} (cf. Cheeger-Gromov \cite{CheegerGromov1990,CheegerGromov1990II}, and Fukaya \cite{Fukaya1987,Fukaya1988,Fukaya1989} ). The existence of such a structure has found many applications in Riemannian geometry (e.g.,\cite{FangRong1999},\cite{PRT1999},\cite{PT1999}, etc. and survey papers \cite{Fukaya2006} and \cite{Rong2007}).

	In general, such structures do not exist on collapsed manifolds of bounded Ricci curvature. Counterexamples were constructed by Anderson \cite{Anderson1992}. It is well-known that many collapsed Ricci-flat manifolds admit no nilpotent structures \cite{GrossWilson2000},\cite{HSVZ2018},\cite{YangLi2019}.
	Finding natural conditions for the existence of a nilpotent structure has been a natural problem in the study of collapsed manifolds under bounded Ricci curvature, around which many useful tools have been established, such as the $\epsilon$-regularity theorem \cite{CheegerTian2006} for Einstein 4-manifolds, smoothing methods under weak harmonic norms \cite{DWY1996},\cite{PWY1999} applicable for points where the conjugate radius has a definite positive lower bound, and $\epsilon$-regularity for points in a collapsed Einstein $n$-manifold where the local fundamental group admits a full rank over a regular space \cite{NaberZhang2016}. Recently, it was known by \cite{HKRX2020} and \cite{HW2020-2} that nilpotent structures exist around points in an $n$-manifold of Ricci curvature $\Ric\ge -(n-1)$ where the following \emph{$(\delta,\rho)$-Reifenberg local covering geometry} condition is satisfied for a small constant $\delta=\delta(n)$ depending only on the dimension.
	
	A point $x$ in an $n$-manifold $(M,g)$ of $\Ric(M,g)\ge -(n-1)$ is said to admit \emph{$(\delta,\rho)$-Reifenberg local covering geometry}, if
	the Gromov-Hausdorff distance,
	\begin{equation}\label{def-Reifenberg}
	d_{GH}\left(B_{r}(\tilde x), B_{r}^{n}(0)\right)\leq \delta \cdot r, \quad 0<r\leq \rho,
	\end{equation}
	where $\tilde x$ (called a $(\delta,\rho)$-Reifenberg point) lies in the preimage of $x$ in the (incomplete) Riemannian universal cover $\pi: (\widetilde{B_{\rho}(x)},\tilde x)\to (B_{\rho}(x),x)$ and $B_{r}^{n}(0)$ denotes an $r$-ball in the $n$-dimensional Euclidean space $\mathbb R^n$.

	Conversely, it was proved by \cite{CFG1992} that if $(M,g)$ is a complete Riemannian $n$-manifold with sectional curvature bound $|\sec(M,g)|\le 1$, then there are positive constants $\delta=\delta(n)>0$ and $\rho=\rho(n)>0$ such that $(M,g)$ has $(\delta,\rho)$-Reifenberg local covering geometry, i.e., \eqref{def-Reifenberg} holds for any point $x\in M$. For a simple proof, see \cite{Rong2020,Rong2021}.

As already mentioned above, a closed collapsed Riemannnian $n$-manifold $(M,g)$ with $|\Ric(M,g)|\leq n-1$ and $(\delta,\rho)$-Reifenberg local covering geometry admits nilpotent structures. That is, for $\delta=\delta(n)$,
$(M,g)$ can be smoothed to a collapsed metric $g(t)$ with bounded sectional curvature $K(t)$ (\cite{DWY1996}, \cite{PWY1999}, \cite{HKRX2020}, cf. \cite{Hamilton1982}, \cite{Perelman2002}) and hence a nilpotent Killing structure associated with the smoothed metric $g(t)$ exists in the sense of Cheeger-Fukaya-Gromov \cite{CFG1992}.

In this paper we consider the uniqueness of those nilpotent structures with respect to the original metric. Compared with the situation of bounded sectional curvature \cite{CFG1992}, where the nilpotent structure are known to be canonical, according to \cite{DWY1996}, \cite{PWY1999} $g(t)$ generally may be not $C^1$-close to the original metric $g$, and the sectional curvature bound $K(t)$ of $g(t)$ may blow up as $g(t)$ approaches $g$. It is natural to ask:
	\begin{problem}\label{problem-uniqueness-of-nilpotent-structure}
		Let $(M_i,g_i)$ be a sequence of Riemannian $n$-manifolds of $|\Ric(M_i,g_i)|\leq n-1$ with $(\delta,\rho)$-Reifenberg local covering geometry for some $\rho>0$ and small $\delta>0$ depending on $n$. Assume that $(M_i,g_i)$ collapse to a compact metric space,
		$$(M_i,g_i)\overset{GH}{\longrightarrow}X.$$
		Is there a canonical nilpotent structure on $(M_i,g_i)$ over $X$ that is uniquely determined by $g_i$ for any fixed large $i$?
	\end{problem}

The purpose of the paper is to answer Problem \ref{problem-uniqueness-of-nilpotent-structure} affirmatively in the case that $X$ is regular, i.e.,
every tangent cone of $X$ is isometric to the Euclidean space $\mathbb R^m$ with $m<n$.
Hereafter, we call such limit space $X$ an \emph{$m$-regular $(\delta,\rho)$-Reifenberg local covering Ricci limit space}.
\par
In \cite[Theorem 0.1]{JKX2022Convergence}, we proved that any regular limit space $X$ in Problem \ref{problem-uniqueness-of-nilpotent-structure} is a $C^{1,\alpha}$-Riemannian manifold $(X,h)$, i.e., $X$ is a $C^{\infty}$-smooth manifold with a $C^{1,\alpha}$-Riemannian metric $h$, whose $C^{1,\alpha}$-harmonic radius depends on the volume of local balls.
Then a nilpotent structure on $(M_i,g_i)$ consists of two ingredients. One is the underlying fibration whose fibers  are  infra-nilmanifolds. The other is the symmetry arising from an action by a simply connected nilpotent Lie group on the Riemannian universal cover $\widetilde{U}$ of a neighborhood $U$ around each point in $M_i$, which extends the deck-transformation on $\widetilde{U}$ and is isometric with respect to a nearby metric.
\par
In \cite[Theorem 0.9]{JKX2022Convergence}, we have constructed a smooth fibration from $(M_i,g_i)$ to $(X,h)$ with the same regularities (i.e., \eqref{thm-fibration-1} and \eqref{thm-fibration-3} below) as the fibrations under bounded sectional curvature constructed in \cite{Fukaya1987},\cite{CFG1992} for each $i$ large.
	\begin{theorem}[\cite{JKX2022Convergence}]\label{thm-fibration}
		Given $\rho,v>0,$ and positive integers $n, m(\leq n)$, there exist positive constants $\delta(n),$  $\epsilon_0=\epsilon_0(n,\rho,v)$ and $C(n,\rho,v)$ such that for $\delta=\delta(n)$ and any $0<\epsilon\leq \epsilon_0$, the following holds.
		
		\par		
			Let $(M,g)$ be a closed Riemannian $n$-manifold with $|\Ric(M,g)|\leq n-1$ and $(\delta,\rho)$-Reifenberg local covering geometry. Assume that $(X,h)$ is a compact $m$-regular $(\delta,\rho)$-Reifenberg local covering Ricci limit space where every $1$-ball's volume $\ge v>0$.
		If $d_{GH}((M,g),(X,h))\le \epsilon $, then there is a $C^\infty$-smooth fibration $f:M\to X$ satisfying
			\begin{enumerate}\numberwithin{enumi}{theorem}
			\item\label{thm-fibration-1} $f$ is a $\varkappa(\epsilon\,|\,n)$-almost Riemannian submersion, i.e., for any vector $\xi$ perpendicular to an $f$-fiber,
			$e^{-\varkappa(\epsilon\,|\,n)}|\xi|_{g}\le |d f(\xi)|_h\leq e^{\varkappa(\epsilon\,|\,n)}|\xi|_{g},$
			where after fixing $n$, $\varkappa(\epsilon\,|\,n)\to 0$ as $\epsilon\to 0$.
			\item\label{thm-fibration-2} The intrinsic diameter of any $f$-fiber $F_{p}=f^{-1}(p)$ over $p\in X$ satisfies $\diam_{g}(F_p) \le C(n,\rho,v)\epsilon.$
			\item\label{thm-fibration-3} The second fundamental form of $f$ satisfies
			$\left|\nabla^2f\right|\leq C(n,\rho,v).$
			\item\label{thm-fibration-4} $F_p$ is diffeomorphic to an infra-nilmanifold.
		\end{enumerate}
	\end{theorem}
The fibration in Theorem \ref{thm-fibration} is construced via gluing Cheeger-Colding's $\delta$-splitting maps \cite{CC1996,CC1997I} together; see \cite{JKX2022Convergence}. According to \cite{CFG1992},  \eqref{thm-fibration-1}-\eqref{thm-fibration-4} are crucial in describing the geometric relation between $(M,g)$ and $(X,h)$.

On the other hand, by the smoothing methods (e.g., Hamilton's Ricci flow \cite{Hamilton1982} in \cite{DWY1996} and embedding to Hilbert spaces by PDEs \cite{Abresch1988} in \cite{PWY1999}), there is $t_0=t_0(n,\rho)>0$ and a positive function $K(t)$ (depending on the smoothing methods) such that for any $0<t\le t_0$, there are nearby metrics $g(t)$ on $(M,g)$ in Theorem \ref{thm-unique-fibration} that are $e^t$-bi-Lipschitz equivalent to $g$ and admit a uniformly bounded sectional curvature $K(t)$ depending on $t$, where $K(t)\to +\infty$ as $t\to 0$. By blowing up the metric $K(t)g(t)$ and by the fibration theorems under $|\sec|\le 1$ in \cite{Fukaya1987},\cite{CFG1992}, one derives
fibrations $f_t$ from $(M,K(t)g(t))$ to $X$ with the regularities \eqref{thm-fibration-1}-\eqref{thm-fibration-4} satisfied by $K(t)g(t)$ and a new metric on $X$, provided that  $(M,g)$ is sufficiently Gromov-Hausdorff close to $(X,h)$ (depending on $K(t)$).

It should be pointed out that, due to the lack of a global and uniform $C^1$-control between $g$ and $g(t)$ via those smoothing methods above, it cannot be expected that fibrations constructed with respect to the smoothed metrics $g(t)$ still satisfies similar regularities \eqref{thm-fibration-1} and \eqref{thm-fibration-3} with respect to the original metric $g$; see \S3.1 and \S3.2.

The first part of this paper is to show that all fibrations above constructed by different methods are isomorphic to each other, and hence are canonically determined by the original metric. Recall that two fibrations $f_i:(M_i,g_i)\to (X_i,h_i)$ $(i=1,2)$ (denoted as $(M_i,X_i,f_i)$ for convenience) are \emph{isomorphic} if there are diffeomorphisms $\Phi:M_1\to M_2$ and $\Psi:X_1\to X_2$ such that $\Psi\circ f_1=f_2\circ \Phi$.

We first show that fibrations satisfying regularities in Theorem \ref{thm-fibration} are isomorphic to each other.
\begin{theorem}\label{thm-unique-fibration}
	Given $\rho,v>0,$ and positive integers $n, m(\leq n)$, there exist constants $\delta(n)>0$ and $\epsilon_0=\epsilon_0(n,\rho,v)>0$ such that for any $0<\epsilon\leq \epsilon_0$, the following holds.
	\par
	Let $(M,g)$ and $(X,h)$ be as in Theorem \ref{thm-fibration}.
	Then any two fibrations $(M,X,f_j)$ $(j=1,2)$ with
	the regularities \eqref{thm-fibration-1}-\eqref{thm-fibration-3} are isomorphic.
\end{theorem}

Furthermore, by Theorem \ref{thm-stability-nearby-metric} and Lemma \ref{lem-small-diameter} below, once a fibration over $X$ is constructed for a nearby Lipschitz equivalent metric $g(t)$ with a bounded sectional curvature $K(t)$, including those in \cite{DWY1996} and \cite{PWY1999}, it is also isomorphic to that of Theorem \ref{thm-fibration}.
A sufficient condition for the existence of fibrations with respect to a smoothed metric is as follows.

\begin{theorem}\label{thm-unique-smoothed}
	Let $(M,g)$ and $(X,h)$ be as in Theorem \ref{thm-fibration}. Let $g(t), t\in (0,1]$ be a family of smooth metrics on $M$ provided by some smoothing method such that
	\begin{equation}\label{ineq-smoothed-metrics}
	e^{-t}g\le g(t)\le e^tg,\quad\text{and}\quad |\sec(M,g(t))|\le K(t).
	\end{equation} There is $t_0(n,\rho,v)>0$ such that if $K(t)^{1/2}d_{GH}((M,g),(X,h))\le t_0(n,\rho,v)$ and $0<t\le t_0(n,\rho,v)$, then a fibration $f_t:(M,g(t))\to X_t$ with respect to bounded sectional curvature $K(t)$ exists and is isomorphic to $(M,X,f)$.
\end{theorem}

For the two concrete smoothing methods via Ricci flow \cite{DWY1996} and embedding $(M,g)$ into $L^2(M,g)$ \cite{PWY1999}, we prove that if both fibrations in Theorem \ref{thm-fibration} and for smoothed metrics $g(t)$ exist, where $g(t)$ may be definite away from $(M,g)$, then they are also isomorphic to each other; see \S3.1 and \S3.2 below.

The main goal of the second part of this paper is to prove there is a canonical nilpotent structure with respect to the original metric $(M,g)$ in Problem \ref{problem-uniqueness-of-nilpotent-structure}.
Let us recall that Cheeger-Fukaya-Gromov \cite{CFG1992} constructed a canonical nilpotent structure $(\mathfrak n, g_{\epsilon})$ on $(M,g)$ with $|\sec(M,g)|\le 1$, such that $\mathfrak n$ is a sheaf of nilpotent Lie algebra generated by vector fields, which is Killing for a $C^1$-nearby metric $g_{\epsilon}$. Moreover, the infinitesimal actions of $(\mathfrak n,g_{\epsilon})$ can be lifted to $g_{\epsilon}$-isometric actions by a simply connected nilpotent Lie group $\mathcal{N}$ on the universal cover of some neighborhood at any point, whose orbit space is $e^\epsilon$-bi-Lipschitz equivalent to $X$. Each orbit of $\mathfrak{n}$, which is defined to be the union of points along integral curves of local vector fields in $\mathfrak{n}$, coincides with the projection of $\mathcal{N}$'s orbit from the local universal cover. 

For $(M,g)$ in Problem \ref{problem-uniqueness-of-nilpotent-structure}, let $g(t)$ be the smoothed metric on $M$ via Ricci flow (see Theorem \ref{thm-smoothing-ricci-flow}). Then $|\Ric(M,g(t))|\le 2(n-1)$. Following Cheeger-Fukaya-Gromov \cite{CFG1992}, we will construct a niloptent Killing structure $(\mathfrak n_{t_0}, \bar g_{t_0})$ from $g(t_0)$ for some $t_0=t_0(n,\rho,v)>0$, whose underlying fibration is exactly that of $(M,g(t_0))$ provided by Theorem \ref{thm-fibration}. Moreover, for any $(M,g)$  close enough to $(X,h)$, $\bar g_{t_0}$ also admits $|\Ric(M, \bar g_{t_0})|\le 3(n-1)$ and a uniform Reifenberg local covering geometry.

\begin{theorem}\label{thm-existence-canonical-nilstr}
	There are $0<t_0(n,\rho,v),\epsilon_0(n,\rho,v)\le 1$ and $\rho_1(n,\rho)>0$ such that for $(M,g)$ and $(X,h)$ in Theorem \ref{thm-fibration}, if $d_{GH}((M,g),(X,h))\le \epsilon\le \epsilon_0$, then there is a nilpotent Killing structures $(\mathfrak n_{t_0}, \bar g_{t_0})$ on $(M,g)$ from $g(t_0)$, whose underlying fibration coincides with that provided by Theorem \ref{thm-fibration}, where $g(t)$ is the solution of Ricci flow equation with initial value $g$. Moreover, there is $\epsilon_1(t_0,n,\rho,v)>0$ such that if $d_{GH}((M,g),(X,h))\le \epsilon_1$, then $|\Ric(M,\bar{g}_{t_0})|\leq 3(n-1)$ and $(M,\bar{g}_{t_0})$ has $(3\delta,\rho_1)$-Reifenberg local covering geometry.
\end{theorem}

Next, we will prove in the last main result (Theorem \ref{thm-uniqueness-canonical-nilstr} below) that
all nearby metrics $g_{\eta}$ on $M$, if it arises a nilpotent Killing structure $(\mathfrak{n}_\eta,g_\eta)$ under a sectional curvature bound, whose quotient space is almost isometric to $(X,h)$, then $(\mathfrak{n}_\eta, g_\eta)$ is isomorphic to the nilpotent structure $(\mathfrak n_{t_0}, \bar g_{t_0})$ in Theorem \ref{thm-existence-canonical-nilstr}.

We say that two nilpotent Killing structures  $(\mathfrak{n}_i,\bar g_i), i=1,2,$ on $M$ with regular quotient spaces $X_i=(M, \bar g_i)/\mathfrak{n}_i$ are equivalent
if there exists
a bundle isomorphism $(\Phi, \Psi)$ between their underlying fibrations $(M,X_i,f_i)$ that also preserves $\mathfrak{n}_i$. It is equivalent to that, for any $x\in M,$ there exists a neighborhood $U_x$ such that the two actions induced by $\mathfrak{n}_i$ are conjugate by the lifting of $\Phi$ on the universal cover of $U_x.$

\begin{theorem}\label{thm-uniqueness-canonical-nilstr}
	There exist $0<\eta_0(n,\rho,v),\epsilon_0(n,\rho,v)\leq 1$ such that the following holds.
	
	Let $(M,g)$ and $(X,h)$ be as in Theorem \ref{thm-fibration} such that $d_{GH}((M,g),(X,h))\le \epsilon\le \epsilon_0$. Let $(\mathfrak{n}_\eta,g_\eta)$ be a nilpotent Killing structure on $M$ with respect to the sectional curvature bound of $g_\eta$, whose quotient space by $\mathfrak{n}_{\eta}$ is $(X,h_\eta)$, such that for $0<\eta\le \eta_0$,
	\begin{equation}\label{bi-lip-equiv}
	e^{-\eta} g\le g_\eta \le e^\eta g,\qquad e^{-\eta} h\le h_\eta \le e^\eta h,
	\end{equation}
	and after rescaling $g_\eta$ to admit $|\sec|\le 1$, the diameter of each $\mathfrak{n}_\eta$'s orbit $\le \epsilon_0$. Then $(\mathfrak{n}_\eta,g_\eta)$ is equivalent to the nilpotent Killing structure $(\mathfrak n_{t_0}, \bar g_{t_0})$ constructed in Theorem \ref{thm-existence-canonical-nilstr}.
\end{theorem}

By Theorem \ref{thm-uniqueness-canonical-nilstr}, an interesting corollary is that, the nilpotent structure via a smoothed metric $g_\eta$ does not depend on the explicit sectional curvature bound of $g_\eta$. In particular, it shows that the nilpotent Killing structure $(\mathfrak n_{t_0}, \bar g_{t_0})$ constructed in Theorem \ref{thm-existence-canonical-nilstr} is uniquely determined by the original metric $g$, hence ``canonical".

Compared with the local compatibility of nilpotent structures in \cite{CFG1992}, Theorem \ref{thm-uniqueness-canonical-nilstr} deals with the case that, after rescaling to $|\sec|\le 1$, the diameter of $(M,K_\eta g_\eta)$ admits no uniform upper bound, where $K_\eta=\max\{\max|\sec(M, g_\eta)|, 1\} $. Here, the local Gromov-Hausdorff approximation (see Lemma \ref{lem-small-diameter} and paragraph before Lemma \ref{lemma-regularity-DWY} below) introduced from \cite{HKRX2020} plays an essential role.


The rest of the paper is organized as follows:
Sections \ref{proof-canonical-fibration} and \ref{Closeness-of-vertical-distributions} are devoted to the proof of Theorem \ref{thm-unique-fibration}, where the main technical results are two stability Theorems  \ref{thm-stability-nearby-metric} and \ref{thm-stability-lipequiv-metric}.
In Section \ref{Closeness-of-vertical-distributions}, we establish the closeness of vertical distributions which is required in the proofs of stability theorems.
In Section \ref{proof-DWY-unique-structure} we study the fibrations via smoothing methods and prove Theorem \ref{thm-unique-smoothed}.
Theorems \ref{thm-existence-canonical-nilstr} and \ref{thm-uniqueness-canonical-nilstr} are proved in Section \ref{unique-nilpotent-structure}, where canonical nilpotent Killing structures are  constructed and proved to be locally stable under nearby perturbation on the metric.
\par
The following notations will be used later in this paper.

$\bullet$\quad A map between two length metric spaces, $f:X\to Y,$ is called an $\epsilon$-Gromov-Hausdorff
approximation, briefly an $\epsilon$-GHA, if $f$ is an $\epsilon$-isometry and $f(X)$ is $\epsilon$-dense in $Y.$ A
sequence of Riemannian $n$-manifolds $M_i$ converges to a metric space $X$ in the Gromov-Hausdorff topology, $M_i\stackrel{d_{GH}}{\longrightarrow}X,$ if and only if there is a sequence of $\epsilon_i$-GHA from $M_i$ to $X, \epsilon_i\to 0$ as $i\to \infty.$

$\bullet$\quad $\varkappa(\epsilon\,|\,a,b,c,\dots)$ denotes a positive  function depending on $\epsilon,a,b,c,\dots$ such that after fixing $a,b,c,\dots$, $\varkappa(\epsilon\,|\,a,b,c,\dots)\to 0$ as $\epsilon\to 0$.

$\bullet$\quad $B_r(x,g)$ denotes the geodesic $r$-ball centered at $x$ with respect to metric $g.$

$\bullet$\quad $\diam_{g}f^{-1}(p)$ denotes the intrinsic diameter of $f^{-1}(p)$ under metric $g.$

$\bullet$\quad $\mathcal{V}_f$ (resp. $\mathcal{H}_f$) denotes the vertical (resp. horizontal) distribution associated to a (not necessarily Riemannian) submersion $f:(M,g)\to (N,h)$.
$\mathcal V_f(x)$ (resp. $\mathcal{H}_f(x)$) denotes the vertical (resp. horizontal) subspace at $x\in M$. $|\nabla^{2}f|_{g,h}$ denotes the $C^0$-norm of the second fundamental form of map $f$ measured in metrics $g$ and $h$.

$\bullet$\quad
The second fundamental form $\sndf_{f}$ of $f$-fibers and integrability tensor $A_f$ of $f$ are defined, respectively, to be the following tensors
\begin{align*}
&\sndf_{f}:\mathcal V_f(x)\times \mathcal V_f(x)\to \mathcal{H}_f(x), \quad \ \sndf_{f}(T,T)=(\nabla_{T}T)^{h}|_{x}\in \mathcal{H}_f(x),\\
&A_f:\mathcal H_f(x)\times \mathcal H_f(x)\to \mathcal V_f(x), \quad A_{f}(X,Y)=[X,Y]^v|_x\in \mathcal V_f(x).
\end{align*}
And $\left|\sndf \right|_{f^{-1}(p),g}$, $\left|A \right|_{f^{-1}(p),g}$ denote the norm of the tensors $\sndf_f$ and $A_f$ restricted on $f^{-1}(p)$ measured in $g$.

\par \textbf{Acknowledgment.}
Z. J. is supported partially by China Postdoctoral Science Foundation Grant No. 8206300494.
S. X. is supported in part by Beijing Natural Science Foundation Grant No. Z190003 and National Natural Science Foundation of China Grant No. 11871349.

\section{Stability of almost Riemannian submersions} \label{proof-canonical-fibration}
In this section we prove stability results on fibrations $(M,X_i,f_i)$ ($i=1,2$) that are almost Riemannian submersions satisfying certain regularities (see below) with respect to two $C^0$-nearby or Lipschitz equivalent metrics $g_i$.
Since we are dealing with different metrics $g_i$ on $M$, it is natural and also useful for applications to consider the case that base spaces $X_i$ may be different.

\begin{theorem}\label{thm-stability-nearby-metric}
	Given any $\rho,v>0$ and positive integers $n,m(\leq n)$, there are constants $\delta(n)>0$, $\epsilon_0=\epsilon_0(n,\rho,v)>0$ and $\eta_0=\eta_0(n,\rho,v)>0$ such that for any $0<\epsilon\leq \epsilon_0$ and $0\leq \eta\leq \eta_0$, the following holds.
\par
Let $(M,g_i)$ and $(X_i,h_i)$ be as in Theorem \ref{thm-fibration} and let $f_i:(M,g_i)\to (X_i,h_i)$ be two fibrations satisfying the following regularities.
\begin{enumerate}\numberwithin{enumi}{theorem}	
	\item\label{thm-stability-almost-Riem-submerison} $f_i$ is an $\epsilon$-almost Riemannian submersion;
	\item\label{thm-stability-diameter-of-fibres}  $\diam_{g_i}(F_{p_i})\leq \epsilon$ for any $f_i$-fiber $F_{p_i}=f_i^{-1}(p_i)$ over $p_i\in X_i$;
	\item\label{thm-stability-second-fundamental-form-diameter}
	$\left|\sndf\right|_{F_{p_i},g_i}\cdot \diam_{g_i}(F_{p_i})\leq \epsilon$, and
	$\left|A\right|_{F_{p_i},g_i}\cdot \diam_{g_i}(F_{p_i})\leq \epsilon.$
\end{enumerate}
If $g_1$ and $g_2$ are $e^{\eta}$-Lipschitz equivalent, i.e.,
	\begin{equation}\label{C0-close-metrics}
		e^{-\eta}g_2\le g_1\le e^{\eta}g_2,
	\end{equation}
then $(M,X_1,f_1)$ is isomorphic to $(M,X_2,f_2)$ by a pair $(\Phi,\Psi)$ of diffeomorphisms
$$
\begin{CD}
	(M,g_1) @>\Phi>> (M,g_2)\\
	@Vf_1VV @Vf_2VV\\
	(X_1,h_1)@>\Psi>>(X_2,h_2)
\end{CD}
$$
such that $\Psi$ is $4e^\eta$-bi-Lipschitz and $\Phi$ is $e^{\varkappa(\epsilon\,|\,n,\rho,v)+\eta}$-bi-Lipschitz.
\end{theorem}

In the case of Theorem \ref{thm-unique-fibration}, \eqref{thm-fibration-1}-\eqref{thm-fibration-3} imply \eqref{thm-stability-almost-Riem-submerison}-\eqref{thm-stability-second-fundamental-form-diameter}.
Compared with \eqref{thm-fibration-3}, the weaker condition \eqref{thm-stability-second-fundamental-form-diameter} is rescaling invariant and independent of the base space's metric tensor $h$.
Hence,
assuming the Theorem \ref{thm-stability-nearby-metric},
Theorem \ref{thm-unique-fibration} holds as a special case.

\begin{proof}[Proof of Theorem \ref{thm-unique-fibration}]
	~
	
	Let $(M,X,f_i),i=1,2$ be any two fibrations as in Theorem \ref{thm-unique-fibration}. Taking $\eta=0, (X_1,h_1)=(X_2,h_2)=(X,h)$.
	For sufficiently small $\epsilon, $ the maximal one of $\varkappa(\epsilon\,|\,n)$, $\epsilon$ and  $C(n,\rho,v)\epsilon$ in regularities \eqref{thm-fibration-1} and \eqref{thm-fibration-2} respectively
	is smaller than the $\epsilon_0$ in Theorem \ref{thm-stability-nearby-metric}.
	Therefore $(M,X,f_1)$ is isomorphic to $(M,X,f_2)$ by Theorem \ref{thm-stability-nearby-metric}.
\end{proof}

Theorem \ref{thm-stability-nearby-metric} is a natural extension of the stability results in \cite{CFG1992} and \cite{PRT1999}. Recall that the stability theorem for fibrations under bounded sectional curvature  is proved in \cite[Appendix 2]{CFG1992}, where the second fundamental form of fibrations are uniformly bounded and with respect to one metric. It play an important role in the construction of nilpotent Killing structure. The stability of fibrations arising from a continuous family of metrics with bounded sectional curvature is established and applied in \cite{PRT1999}. Compared with the earlier known stability results under bounded sectional curvature,  other differences here are that the second fundamental form of maps in Theorem \ref{thm-stability-nearby-metric} may blow up,  and the base spaces is weaken from smooth Riemannian manifolds to  $C^{1,\alpha}$-Ricci limit spaces.

In the study of Theorem \ref{thm-unique-smoothed}, one also encounters the case that two metrics are not $C^0$-close enough to each other. Then we need another version of the stability for almost Riemannian submersions associated to two Lipschitz equivalent metrics (cf. \cite{JiangXu2019}).

\begin{theorem}\label{thm-stability-lipequiv-metric}
	Given any $L_0\ge 1, \rho, v>0$ and positive integers $n,m(\leq n)$, there are constants $\delta(n)>0$ and $\epsilon_0=\epsilon_0(L_0, n,\rho,v)>0$ such that for any $0<\epsilon\leq \epsilon_0,$  the following holds.
\par
Let $(M,g_i)$ and $(X_i,h_i)$ be as in Theorem \ref{thm-fibration} and let $f_i:(M,g_i)\to (X_i,h_i)$ be two fibrations with the regularities \eqref{thm-stability-almost-Riem-submerison}-\eqref{thm-stability-second-fundamental-form-diameter}. If the metrics $g_i$ and their Levi-Civita connections satisfy
\begin{equation}\label{C1-control-metrics}
L_0^{-1} g_2\le g_1\le L_0 g_2, \qquad |\nabla^{g_1}-\nabla^{g_2}|_{g_1}\leq L_0,
\end{equation}
then $(M,X_1,f_1)$ is isomorphic to $(M,X_2,f_2)$ by a pair $(\Phi,\Psi)$ of diffeomorphisms such that $\Psi$ is $4L_0$-bi-Lipschitz and $\Phi$ is $e^{\varkappa(\epsilon\,|\,L_0,n,\rho,v)}L_0$-bi-Lipschitz.
\end{theorem}

But up to a suitable rescaling, the condition  \eqref{C1-control-metrics} can be replaced by
\begin{equation}\label{nearby-connections}
L_0^{-1} g_2\le g_1\le L_0 g_2, \qquad |\nabla^{g_1}-\nabla^{g_2}|_{g_1}\leq \epsilon.
\end{equation}
\par	

The rest of this section is to prove Theorems \ref{thm-stability-nearby-metric} and \ref{thm-stability-lipequiv-metric}.
Though the condition \eqref{C1-control-metrics} can be reduced to \eqref{nearby-connections}, we will do this only when they are necessarily required.

We first prove Theorem \ref{thm-stability-lipequiv-metric} under the assumption that $L_0^{-1} g_2\le g_1\le L_0 g_2$, most of whose arguments are also suitable for Theorem \ref{thm-stability-nearby-metric}.

Let $(M,g)$ be a closed Riemannian $n$-manifold with $|\Ric_{M}|\leq n-1$ and $(\delta,\rho)$-Reifenberg local covering geometry for some $\delta=\delta(n)>0$. Via Dai-Wei-Ye \cite{DWY1996}, there is uniform flow time for Ricci flow equation with initial value $g$ (see Theorem \ref{thm-smoothing-ricci-flow} below). The following smoothing result on the limit space will be applied.
\begin{theorem}[{\cite[Corollary 0.2]{JKX2022Convergence}}]\label{thm-smoothing}
Let $(M_i,g_i)$ be a sequence of closed Riemannian $n$-manifolds with $|\Ric_{M_i}|\leq n-1$ and $(\delta,\rho)$-Reifenberg local covering geometry for some $\delta=\delta(n)$ such that $(M_i,g_i)\stackrel{GH}{\longrightarrow} (X,h)$, where $(X,h)$ is an $m$-regular Ricci limit space where $B_1(p)$ has volume $\ge v>0$ for any $p\in X$.
		Let $g_i(t), 0<t\leq  T(n,\rho)$ be the solutions of Ricci flow equation with initial value $g_i(0)=g_i.$
		Then passing to a subsequence $(M_i, g_i(t))\stackrel{GH}{\longrightarrow}(X,h(t)),$ where $\{h(t), 0<t\leq  T(n,\rho)\}$ is a smooth family of  Riemannian metrics on $X$ satisfying
	\begin{enumerate}[(a)]
		\item\label{full-rank1a} $h(t)$ converge to $h$ in the $C^{1,\alpha}$-topology as $t\to 0$ for any $1>\alpha>0.$
		\item\label{full-rank1b} the sectional curvature of $h(t)$ satisfies $|\sec(X, h(t))|\le C(n,\rho,v)t^{-\frac{1}{2}}$,
	    \item \label{full-rank1c} $e^{-2t}d_{h}(p,q)\leq d_{h(t)}(p,q)\leq e^{2t}d_{h}(p,q)$ for any $p,q\in X$,
\end{enumerate}
	where $d_{h}$ (resp. $d_{h(t)}$) is the distance function induced by $h$ (resp. $h(t)$) and the constants $T(n,\dots), C(n,\dots)$ depend only on the given parameters.
\end{theorem}

\begin{proof}[Proof of Theorem \ref{thm-stability-lipequiv-metric}]	
	\item \indent\indent The proof of Theorem \ref{thm-stability-lipequiv-metric} can be divided into two steps.
    \item \indent\indent Step 1. To reduce the Theorem \ref{thm-stability-lipequiv-metric} to the case that the base space $(X_i,h_i)$ admits $|\sec(X_i,h_i)|\leq 1$ and $\injrad(X_i,h_i)\geq i_0$ for some $i_0>0$  by Theorem \ref{thm-smoothing}.
\par
Let $f_i:(M,g_i)\to (X_i,h_i)$($i=1,2$) be two fibrations as in Theorem \ref{thm-stability-lipequiv-metric}.
	By Theorem \ref{thm-smoothing}, there are nearby smooth metrics $h_i(t), t\in (0, T(n,\rho)]$ satisfying $|\sec(X_i,h_i(t))|\leq C(n,\rho,v)t^{-1/2}$ for some constant $C(n,\rho,v)$.	
	 For any $0<\epsilon\le T(n,\rho)$, let $\bar h_i=\epsilon^{-1}h_i(\epsilon)$ and $\bar g_i=\epsilon^{-1}g_i$. Then
	\begin{equation}\label{condition-bounded-Ricci-sec}
	\begin{aligned}
	|\Ric(M,\bar g_i)|\le (n-1)\epsilon, \ |\sec(X_i,\bar h_i)|\le C(n,\rho,v)\epsilon^{1/2}.
	\end{aligned}
	\end{equation}
 Since $h(\epsilon)$ is $C^{1,\alpha}$-close to $h$, we have 
 $\vol(B_1(p_i, \bar h_i))\geq v_0(n,v)>0, \forall p_i\in X_i.$  By the injectivity radius estimate \cite[Theorem 4.3]{CGT1982}) under bounded sectional curvature, one has $\injrad(X_i, \bar{h}_i)\ge i_0(n,\rho,v)>0$.
\par
Simultaneously, $f_i:(M,\bar g_i)\to (X_i,\bar h_i)$  has regularities  as follows: $f_i$ is
$\varkappa(\epsilon\,|\,n)$-almost Riemannian submersion and the intrinsic diameter of $f_i$-fiber measured by the metric $\bar{g}_i$ satisfies $\diam_{\bar{g}_i}f_{i}^{-1}(p_i)\leq \epsilon^{1/2}$ for any $p_i\in X_i$.
By the scaling invariance, $f_i:(M,\bar g_i)\to (X_i,\bar h_i)$ satisfies (\ref{thm-stability-second-fundamental-form-diameter}) and  $\bar{g}_i$ $(i=1,2)$ are $L_0$-equivalent to each other.
Without loss of generality, we still assume that $f_i$ has regularities \eqref{thm-stability-almost-Riem-submerison}-\eqref{thm-stability-second-fundamental-form-diameter}. Therefore, Step 1 holds.

\item \indent\indent Step 2. Assuming $(X_i,h_i)$ satisfies $|\sec(X_i,h_i)|\leq 1$ and $\injrad(X_i,h_i)\geq i_0>0$, to construct diffeomorphisms $\Psi:X_1\to X_2$ and $\Phi:M\to M$ such that $\Psi\circ f_1=f_2\circ \Phi$.
\par
 The constructions of $\Phi$ and $\Psi$ are similar to that in \cite{JiangXu2019}. In the following,  we sketch out the constructions by pointing out the key points and refer to \cite{JiangXu2019} for the detailed arguments.
 \par
 $\bullet$ To construct a $4L_0$-bi-Lipschitz diffeomorphism $\Psi: X_1\to X_2$ such that $f_1$ and $\Psi^{-1}\circ f_2$ are fiberwisely close.
 \par
 For any $p\in X_1$, let $F_{1,p}=f_{1}^{-1}(p)$. Via the regularities \eqref{thm-stability-almost-Riem-submerison}, \eqref{thm-stability-diameter-of-fibres} and (\ref{C1-control-metrics}), one has
$\diam_{h_2}(f_{2}(F_{1,p}))\leq L_{0}^{1/2}e^{\epsilon}\epsilon.$
 Since $(X_{2},h_2)$ satisfies $|\sec(X_2,h_2)|\leq 1$ and $\injrad(X_2,h_2)\geq i_0>0$, we define $\Psi(p)$ to be the center of mass of $f_2(F_{1,p})\subset (X_2,h_2)$.
 For any $v\in \mathcal{H}_{f_1}, $ by calculating $d\Psi(df_1(v))$ is represented by
 the integral of $df_2(v)$ subtract an item determined by
 $H _{F_{1,p}} \cdot {\rm diam_{g_1}}(F_{1,p})$ (see \cite[(5.1.6)]{JiangXu2019}), where $H_{F_{1,p}}$ is the mean curvature of fiber $F_{1,p}$.
 Observe that if  for any unit vector $v\in \mathcal{H}_{f_1}, $ $df_2(v)$ is definitely away from zero and $\left|\sndf\right|_{F_{1,p},g_1}\cdot {\rm diam_{g_1}}(F_{1,p})$ is sufficient small, then $d\Psi$ is non-degenerate, thus, $\Psi$ is a diffeomorphism,  and by a calculation,  it is $4L_0$-bi-Lipschitz.
 \par
 $\bullet$ To construct a diffeomorphism $\Phi_1:M\to M$ such that $ f_1\circ \Phi_1=\Psi^{-1}\circ f_2$, and $\Phi_1$  is $e^{\varkappa(\epsilon\,|\,L_0,n,\rho,v)}L_0$-bi-Lipschitz. Taking $\Phi=\Phi_1^{-1}$, we have $\Psi\circ f_1=f_2\circ \Phi$.
\par
For any $x\in M,$ denote $p=f_1(x), \Psi(q)=f_2(x).$ Let $\gamma_x:[0,1]\to X_1$ be the unique minimal geodesic with $\gamma_x(0)=p, \gamma_x(1)=q$ and $\tilde \gamma_x:[0,1]\to M$ its unique $f_1$-horizontal lifting at $x.$  Define   $\Phi_1(x)=\tilde \gamma_x(1)$.  Then $\Phi_1$ depends smoothly on $x$ and lies in $F_{1,q}$. Thus $\Phi_1$ is a bundle map from $(M,X_1,\hat f_2=\Psi^{-1}\circ f_2)$ to $(M,X_1,f_1)$, i.e., $ f_1\circ \Phi_1=\Psi^{-1}\circ f_2$.

Since the kernel of $d\Phi_1$ is contained in $\mathcal{V}_{\hat f_2}(x),$  in order to show the non-degeneration of $d\Phi_1,$ it is sufficient to check  for any $g_1$-unit vector $v\in \mathcal{V}_{\hat f_2}(x),$ $d\Phi_1(v)\neq 0.$
Let $\phi: f_1^{-1}(B_{i_0}(p, h_1))\to B_{i_0}(p, h_1)\times F_{1,q}, \  \phi(x)=(f_1(x), \phi_2(x)=\tilde \gamma_x(1))$ be the local trivialization via $f_1$-horizontal lifting geodesics.  Then $d\Phi_1(v)=d\phi_2(v).$
 Let $v^{\top}$ be the orthogonal projection of $v$ on  $\mathcal{V}_{f_1}(x)$ with respect to $g_1.$    An explicit estimation by variation of horizontal curves shows $\left|d\phi_2(v)\right|_{g_1}$ is no less than that $e^{-e^\epsilon \left|\sndf\right|_{F_{1,\gamma},g_1}\cdot {\rm diam_{g_1}}(F_{1,\gamma})}\left|v^{\top}\right|_{g_1}$ subtract  an item determined by   	$\left|\sndf \right|_{F_{1,\gamma}, g_1}\cdot {\rm diam_{g_1}}(F_{1,\gamma})$ and $ \left|A \right|_{F_{1,\gamma}, g_1}\cdot {\rm diam_{g_1}}(F_{1,\gamma})$ (see  \cite[(6.2.2)]{JiangXu2019}). Observe that if $\mathcal{V}_{f_1}$ and $\mathcal{V}_{f_2}$ are close to each other, and $\left|\sndf\right|_{F_{1,\gamma},g_1}\cdot {\rm diam_{g_1}}(F_{1,\gamma})$ and $ \left|A\right|_{F_{1,\gamma},g_1}\cdot {\rm diam_{g_1}}(F_{1,\gamma})$ are sufficiently small, then $d\Phi_1$ is non-degenerate, and the restriction of $\Phi_1$ on each $f_2$-fiber is $e^{\varkappa(\epsilon\,|\,L_0,n,\rho,v)}L_0$-bi-Lipschitz. Moreover, for any $g_{1}$-unit vector $w$ in the horizontal subspace of $f_{1}$, one has $d\Psi\circ df_{1}(w)=df_{2}\circ d\Phi_{1}(w)$. Since $f_{i}:(M,g_i)\to (X_i,h_i)$ are $e^{\varkappa(\epsilon|n)}$-almost Riemannian submersions, $|d\Phi_1(w)|_{g_2}\geq e^{-\varkappa(\epsilon|n)} L_0^{-1}$. By considering the inverse map of $\Phi_1$, which is constructed similar as  $\Phi_1$, we have $|d\Phi_1(w)|_{g_2}\leq e^{\varkappa(\epsilon|n)}L_0$. Hence, $\Phi_1$ is an $e^{\varkappa(\epsilon|L_0, n,\rho,v)}L_0$-bi-Lipschitz diffeomorphism.

\par
A much easier case (e.g., Theorem \ref{thm-unique-fibration}) is that $\left|\nabla^{2}f_{i}\right|_{g_i}$ is bounded by a universal constant $C<+\infty$, $g_i$ are $C^1$-close to each other, and the intrinsic diameters of fibers are small, which by considering the lifting of $f_i$ to local universal covers $\widetilde{B_{\rho}(x)}$ will imply that $f_i$ ($i=1,2$) are $C^1$-close with each other. Then
by the observations above,  $d\Phi_1$ and $d\Psi$ are non-degenerate.  Hence $\Phi_1$ and $\Psi$ will be diffeomorphisms.

Under the weaker conditions  \eqref{thm-stability-second-fundamental-form-diameter} and \eqref{C1-control-metrics},  more efforts are required. The motivation for us to consider \eqref{thm-stability-second-fundamental-form-diameter} is that, it is independent of $h_i$ and by the observations above, it is more natural than \eqref{thm-fibration-3} for non-degeneration of $d\Phi_1$ and $d\Psi$.

Instead of closeness between $df_i$, we will prove only for the vertical distributions of $f_i$.
We will prove that
\begin{enumerate}\addtocounter{enumi}{+3}
	\item\label{prop-vertical-close} For any $x\in M$, the dihedral angle measured in $g_2$ between the subspaces $\mathcal V_{f_1}(x)$ and $\mathcal V_{f_2}(x)$ is no more than $\varkappa(\epsilon\,|\,L_0,n,\rho,v)$.
\end{enumerate}

By the observation on $d\Phi_1$ above, \eqref{prop-vertical-close} together with \eqref{thm-stability-second-fundamental-form-diameter} imply that $d\Phi_1$ is not degenerated.

In order to show that $d\Psi$ is non-degenerate, a definite deviation control between $\mathcal{H}_{f_i}$ is necessary. In fact, by the $L_0$-Lipschitz equivalence between $g_i$ and \eqref{prop-vertical-close}, it is easy to check
\begin{enumerate}\addtocounter{enumi}{+4}
\item\label{prop-diff-close1} There exists a constant $\theta_0=\theta_0(L_0,n,\rho,v)$ satisfying $0\leq \theta_0<\frac{\pi}{2}$ such that the dihedral angle measured in $g_2$ between the subspaces $\mathcal H_{f_1}(x)$ and $\mathcal H_{f_2}(x)$ is no more than $\theta_0$.
\end{enumerate}
Then for any $f_1$-horizontal vector $w\in T_xM$ with $|w|_{g_1}=1$, its orthogonal projection $w_2$ on $\mathcal H_{f_2}$ satisfies $|w_2|_{g_2}\ge |w_1|_{g_2} \cos\theta_0\ge L_0^{-1/2}\cos\theta_0$. Then by the definition of almost Riemannian submersions, $$e^\epsilon L_0^{1/2}\ge |df_2(w_2)|_{h_2}\ge e^{-\epsilon}|w_2|_{g_2}\ge e^{-\epsilon} L_0^{-1/2}\cos\theta_0.$$
By the observation in the construction of $\Psi$, $d\Psi$ is also non-degenerate.

We now prove that \eqref{prop-vertical-close} together with the $L_0$-Lipschitz equivalence between $g_i$ implies \eqref{prop-diff-close1}.

Indeed, for any $W\in \mathcal H_{f_1}(x)$ with a unit $g_2$-norm $|W|_{g_2}=1$, let $W=W_i+W_i'$ be the orthogonal decomposition of $W$ under $g_2$ such that $W_i\in \mathcal V_{f_i}(x)$, $W_2'\in \mathcal H_{f_2}(x)$. Because $\mathcal V_{f_i}(x)$ are close to each other, $|W_1-W_2|_{g_2}=|W_2'-W_1'|_{g_2}=\varkappa(\epsilon\,|\,L_0,n,\rho,v)$. In order to show \eqref{prop-diff-close1}, it suffices to show a positive lower bound on the $g_2$-norm of $W_2'$, which can be see from the calculation below:
$$\begin{aligned}
L_0^{-1}&\le g_1(W,W)=g_1(W,W_1')\le  g_1(W,W_2')+\varkappa(\epsilon\,|\,L_0,n,\rho,v)\\
&\le |W|_{g_1}\cdot|W_2'|_{g_1}+\varkappa(\epsilon\,|\,L_0,n,\rho,v)\le  L_0 |W_2'|_{g_2}+\varkappa(\epsilon\,|\,L_0,n,\rho,v).
\end{aligned}
$$

The verification of  \eqref{prop-vertical-close} is left to the next section; see \eqref{Closeness-of-vertical-distributions-2}.
\end{proof}

\begin{proof}[Proof of Theorem \ref{thm-stability-nearby-metric}]
	~
	
	The proof of Theorem \ref{thm-stability-lipequiv-metric} goes through for Theorem \ref{thm-stability-nearby-metric} if the closeness \eqref{prop-vertical-close} of vertical distributions $\mathcal V_{f_i}$ still holds under \eqref{C0-close-metrics}. It will be verified together with \eqref{prop-vertical-close} in the next section; see \eqref{Closeness-of-vertical-distributions-1}.
\end{proof}

\section{Closeness of vertical distributions}\label{Closeness-of-vertical-distributions}
In this section, we will prove  \eqref{prop-vertical-close}, which is restated as the following.
\begin{proposition}[Closeness of vertical distributions]\label{key-proposition}
	Given any $L_0\ge 1, \rho,v>0$ and positive integers $n\leq m$, there are $\delta(n)>0, \eta_0(n,\rho,v), \epsilon_1(n,\rho,v)>0$ and $\epsilon_2(L_0, n,\rho,v)>0$ such that the following holds.
	\par
	Let $(M,g_i)$ and $(X_i,h_i)$ ($i=1,2$) be as in Theorem \ref{thm-fibration}. Suppose that $f_i:(M,g_i)\to (X_i,h_i)$ are two fibrations with the regularities \eqref{thm-stability-almost-Riem-submerison}, \eqref{thm-stability-diameter-of-fibres} and	
	\begin{equation}\label{second-fundamental-form-diameter-Propo}
		\left|\sndf\right|_{F_{i,p_i},g_i}\cdot \diam_{g_i}(F_{i,p_i})\leq \epsilon,
	\end{equation}
	for any $p_i\in X_i$ and $f_i$-fiber $F_{i,p_i}=f_i^{-1}(p_i)$.
\par	
\begin{enumerate}\numberwithin{enumi}{theorem}
\item\label{Closeness-of-vertical-distributions-1} If $0<\epsilon\leq \epsilon_1$ and $g_i$ ($i=1,2$) satisfy \eqref{C0-close-metrics}, then \eqref{prop-vertical-close} holds for $\varkappa(\epsilon\,|\,n,\rho,v)$.

\item\label{Closeness-of-vertical-distributions-2} If  $0<\epsilon\leq \epsilon_2$  and  $g_i$ ($i=1,2$) satisfy \eqref{C1-control-metrics}, then \eqref{prop-vertical-close} holds.
\end{enumerate}
\end{proposition}

Before giving the proof of Proposition \ref{key-proposition}, we make some preparation.
\par
For any $x\in M$, by the $(\delta,\rho)$-Reifenberg local covering geometry of $g_i$ at $x$, a lift point of $x$ in the universal cover $\widetilde{B_{\rho}(x,g_i)}$ admits a uniform $C^{1,\alpha}$-harmonic radius $\ge r_0(n,\rho,\alpha,Q)>0$, where $0<\alpha<1$ and $Q\ge 1$ are fixed constants in the definition of harmonic radius.
Since $g_i$ are $L_0$-Lipschitz equivalent, $U=B_{\rho}(x,g_1)\cap B_{\rho}(x,g_2)$ contains $B_{L_0^{-1}\rho}(x,g_i)$.
Let $\pi:(\widetilde{U},\tilde x, \tilde{g}_i)\to (U,x,g_i)$ is the Riemannian universal cover of $U=B_{\rho}(x,g_1)\cap B_{\rho}(x,g_2)\subset (M,g_i)$. Then By Anderson's $C^{1,\alpha}$-convergence \cite{Anderson1990} the $C^{1,\alpha}$-harmonic radius of $\tilde x$ in $(\widetilde{U},\tilde{g}_{i})$
\begin{equation}\label{harmonic-radius-bound}
r_{h;\alpha,Q}(\tilde x, \tilde g_i)\geq \min\left\{L_0^{-1}\rho,r_0(n,\rho,\alpha,Q)\right\}>0.
\end{equation}

Since the dihedral angle between two linear spaces is rescaling invariant, after blowing up the metric by $\hat \epsilon\to 0$ (will be explicitly chosen later), it follows from \eqref{harmonic-radius-bound} that
\begin{equation}
\left(\widetilde U, \hat\epsilon^{-1}\tilde{g}_i, \tilde{x}\right) \overset{C^{1,\alpha}}{\longrightarrow} \left(\mathbb{R}^n, g_{i,\infty}, 0_i^n\right)\quad \text{as} \ \epsilon\to 0,
\end{equation}
where $g_{i,\infty}$ are Euclidean metrics.

Note that the $C^{1,\alpha}$-convergence of $\tilde g_i$ is realized by its harmonic coordinates local chart $\alpha_i:(\widetilde{U}, \tilde g_i, \tilde x)\to (\mathbb{R}^n,g_{i,\infty},0_i^n)$, which generally is different for $i=1,2$. In order to deal with them at the same time, we consider the following sub-convergence of identity map $\operatorname{Id}_j:\widetilde{U}_j\to \widetilde{U}_j$ as $\epsilon_j\to 0$:
\begin{equation}\label{blowing-up-epsilon}
\begin{gathered}
\xymatrix{(\widetilde U_j, \hat\epsilon_j ^{-1}\tilde{g}_{1,j}, \tilde{x}_j) \ar[r]_{\alpha_{1,j}}^{C^{1,\alpha}} \ar[d]_{\operatorname{Id}_j} & (\mathbb{R}^n, g_{1,\infty}, 0_1^n) \ar[d]^{I_{\infty}}\\
	(\widetilde U_j, \hat\epsilon_j^{-1}\tilde{g}_{2,j}, \tilde{x}_j) \ar[r]_{\alpha_{2,j}}^{C^{1,\alpha}}&(\mathbb{R}^n, g_{2,\infty}, 0_2^n)
},
\end{gathered}
\end{equation}
\begin{lemma}\label{lem-limit-isomorphism}
	Under either \eqref{Closeness-of-vertical-distributions-1} for $\eta_j\to 0$ or \eqref{Closeness-of-vertical-distributions-2}, any partial limit $I_\infty$ is a linear isomorphism, and $\alpha_{2,j}\circ\operatorname{Id}_j\circ \alpha_{1,j}^{-1}$ (at least) $C^1$-converges to $I_\infty$ as $\epsilon_j\to 0$.
\end{lemma}
\begin{proof}
	For \eqref{Closeness-of-vertical-distributions-1}, \eqref{C0-close-metrics} with $\eta_j\to 0$ implies that $I_\infty$ is isometric, and the harmonic coordinates charts $\alpha_{1,j}$ and $\alpha_{2,j}$ are $\varkappa(\eta\,|\,n)$-$C^{2,\alpha}$-close to each other.
	
	For \eqref{Closeness-of-vertical-distributions-2}, \eqref{nearby-connections} implies that in the coordinates chart $\alpha_{1,j}$, the partial derivative  $\partial_{s}(\alpha_{1,j}^{-1})^*(\hat \epsilon^{-1}\tilde g_{2,j})_{kl}\to 0$ as $\epsilon_j\to 0$. Hence $(\alpha_{1,j}^{-1})^*(\hat \epsilon_j^{-1}\tilde g_{2,j})$ converges to an inner product space with a constant metric tensor, and thus $I_\infty$ is a linear isomorphism.
	
	To see that the differential  $d(\alpha_{2,j}\circ\operatorname{Id}_j\circ \alpha_{1,j}^{-1})$ converges, let us consider a  minimal $\hat \epsilon_{j}^{-1}\tilde g_{1,j}$ geodesic $\gamma_j$. By \eqref{nearby-connections}, $\gamma_{j}$ is almost an $\hat \epsilon_{j}^{-1}\tilde g_{2,j}$-geodesic, in the sense that its 2nd derivative of each component function in the coordinates $\alpha_{2,j}$ goes to $0$ as $\epsilon_j\to 0$. Hence $\gamma_{j}$ $C^1$-converges to a $\gamma_{\infty}$ in
	$(\mathbb{R}^n, g_{1,\infty}, 0_1^n)$, and also $C^1$-converges to $I_\infty(\gamma_\infty)$ in $(\mathbb{R}^n, g_{2,\infty}, 0_2^n)$. Hence $d(\alpha_{2,j}\circ\operatorname{Id}_j\circ \alpha_{1,j}^{-1})$ also converges to $dI_\infty$.
\end{proof}

Next, let us establish the relationship between $f_i:(M,g_i)\to (X_i,h_i)$. Recall that by the construction of $\Psi:X_1\to X_2$, for any $p\in X_1$, $\Psi(p)$ is defined to be the center of mass of $f_2(f_1^{-1}(p))$. Let us consider $\psi:X_1\to X_2$, which is defined by the axiom of choice $\psi(p)\in f_2(f_1^{-1}(p))$. It can be directly checked that
\begin{lemma}[{\cite[\S4]{JiangXu2019}}]\label{lem-psi}
	If $g_1$ and $g_2$ are $L_0$-Lipschitz equivalent, then for any $y\in M$,
	\begin{equation*}
	d_{h_2}(\psi(f_1(y)),f_2(y))\le  e^{\epsilon} \diam_{g_2}(f_1^{-1}(f_1(y)))\le e^{\epsilon} L_0\diam_{g_1}(f_1^{-1}(f_1(y))),
	\end{equation*}
	and for any $p,q\in X_1$,
	\begin{equation*}\label{estimate-psi}
	d_{h_2}(\psi(p),\psi(q)) \le  e^{2\epsilon} L_0d_{h_1}(p,q)+e^{\epsilon} L_0 \min\left\{\diam_{g_1}(f_1^{-1}(p)),\diam_{g_1}(f_1^{-1}(q))\right\}.
	\end{equation*}
\end{lemma}

Now let us specify $\hat \epsilon=\hat \epsilon(x)$ for a point $x\in M$. After a permutation on $i=1,2$, we assume that
\begin{align}\label{blowup-number-def}
\hat \epsilon(x)=& \min\left \{1, \left|\sndf\right|_{f_{1}^{-1}( f_{1}(x)),g_{1}}^{-1}\right\}\cdot \diam_{g_{1}}(f_{1}^{-1}(f_{1}(x)))\\
\le &
\min\left\{1,\left|\sndf\right|_{f_{2}^{-1}( f_{2}(x)),g_{2}}^{-1}\right\}\cdot \diam_{g_{2}}(f_{2}^{-1}(f_{2}(x))).\nonumber
\end{align}
Then by \eqref{thm-stability-diameter-of-fibres} and \eqref{thm-stability-second-fundamental-form-diameter}, after blowing up we have
\begin{equation}\label{estimate-2nd-fund-form-blowup}
\left|\sndf\right|_{f_{i}^{-1}( f_{i}(x)),\hat \epsilon^{-1}g_{i}}\le \epsilon^{1/2}, \qquad (i=1,2)
\end{equation}
and
\begin{equation}\label{estimate-fiber-diameter-blowup}
\diam_{\hat \epsilon^{-1}g_{1}}(f_{1}^{-1}(f_{1}(x)))\le \epsilon^{1/2}.
\end{equation}

Moreover, for any fixed $R>0$
and $q\in B_{R}(f_{1}(x),\hat{\epsilon}^{-1}h_{1})$, let $\gamma :[0,c]\rightarrow \left(X_{1},\hat \epsilon^{-1}h_{1}\right)$ be a unit speed minimal geodesic from $f_{1}(x)$ to $q$.  By a standard variation on the $f_{1}$-horizontal lifting curves of $\gamma(t)$ (cf. \cite[Lemma 1]{LiXu18}) and \eqref{second-fundamental-form-diameter-Propo},
$$\frac{\dif}{\dif t}(\diam_{\hat{\epsilon}^{-1}g_{1}} f_{1}^{-1}(\gamma(t)))\le e^{\epsilon} \left|\sndf\right|_{f_{1}^{-1}(\gamma(t)),\hat{\epsilon}^{-1}g_{1}}\cdot \diam_{\hat{\epsilon}^{-1}g_{1}} f_{1}^{-1}(\gamma(t))
\le e^{\epsilon} \epsilon,$$
which implies
\begin{equation}\label{eq-diam-vary}
\left|\diam_{\hat{\epsilon}^{-1}g_{1}}(f_{1}^{-1}(q))-\diam_{\hat{\epsilon}^{-1}g_{1}}(f_{1}^{-1}(f_1(x)))\right|\le e^{\epsilon} \epsilon R.
\end{equation}

Combing \eqref{estimate-fiber-diameter-blowup}, \eqref{eq-diam-vary} and Lemma \ref{lem-psi}, we derive
\begin{lemma}\label{lem-psi-blowup} If $g_1$ and $g_2$ are $L_0$-Lipschitz equivalent, then for any $y\in B_{R}(x,\hat{\epsilon}^{-1}g_{1})$,
	\begin{equation}\label{estimate-psi-fib-blowup}
	d_{h_2}(\psi(f_1(y)),f_2(y))\le e^{\epsilon} L_0(e^\epsilon\epsilon R+\epsilon^{1/2}),
	\end{equation}
	and for any $p,q\in B_{R}(f_{1}(x),\hat{\epsilon}^{-1}h_{1})$,
	\begin{equation}\label{estimate-psi-blowup}
	d_{\hat\epsilon^{-1}h_2}(\psi(p),\psi(q)) \le  e^{2\epsilon} L_0d_{\hat\epsilon^{-1}h_1}(p,q)+e^{\epsilon} L_0 (e^\epsilon \epsilon R +\epsilon^{1/2}),
	\end{equation}
\end{lemma}

We now ready to prove Proposition \ref{key-proposition}.
\begin{proof}[Proof of Proposition \ref{key-proposition}]	
	~
	
Let us argue by contradiction. Assume that there exist $\theta>0$, and contradicting fibrations  $f_{i,j}:(M_j,g_{i,j})\to (X_{i,j},h_{i,j}), i=1,2, j=1,2,\dots$ satisfying the regularities
\eqref{thm-stability-almost-Riem-submerison}, \eqref{thm-stability-diameter-of-fibres} and  \eqref{second-fundamental-form-diameter-Propo} for $\epsilon_j\to 0$, such that either \eqref{C0-close-metrics} holds for $\eta_j\to 0$ as $j\to \infty$, or \eqref{C1-control-metrics} holds, but the dihedral angle between the vertical subspaces $\mathcal V_{f_{i,j}}(x_j)$ of $f_{i,j}$ at some $x_j\in M_j$ is no less than a fixed angle $\theta>0$.

Let $U_{j}=B_{\rho}(x_j,g_{1,j})\cap B_{\rho}(x_j,g_{2,j})$ and let $\pi_j:(\widetilde{U}_{j},\tilde{g}_{i,j},\tilde{x}_j)\to (U_{j},g_{i,j},x_{j})$ be the Riemannian universal cover. Let $\hat \epsilon_j$ be defined by \eqref{blowup-number-def} for $x_j$.
After blowing up by $\hat \epsilon_j^{-1}$, one has
	\begin{equation}\label{C1-convergence-1}
	\left(\widetilde U_j, \hat\epsilon_j ^{-1}\tilde{g}_{i,j}, \tilde{x}_j\right) \overset{C^{1,\alpha}}{\longrightarrow} \left(\mathbb{R}^n, g_{i,\infty}, 0_i^n\right)\quad \text{as} \ j\to \infty,
	\end{equation}
and since the $C^{1,\alpha}$-harmonic radius of $(X_{i,j},h_{i,j})$ have a uniform positive lower bound by $r_0(n,\rho,v,\alpha,Q)$ via \cite[Theorem 0.4]{JKX2022Convergence},
\begin{equation}\label{C1-convergence-2}
\left(X_{i,j},\hat\epsilon_j^{-1}h_{i,j},f_{i,j}(x_j)\right) \overset{C^{1,\alpha}}{\longrightarrow} \left(\mathbb{R}^m, h_{i,\infty},0_i^m\right) \quad \text{as} \ j\to\infty.
\end{equation}

Let $\tilde f_{i,j}=f_{i,j}\circ\pi_j:(\widetilde{U}_j, \hat\epsilon_j^{-1}\tilde g_{i,j},\tilde x_j)\to (X_{i,j},\hat \epsilon_j^{-1} h_{i,j},f_{i,j}(x_j))$ be the lifting maps from $U_j$.
By passing to a subsequence, as $j\to \infty$ for each $i=1,2$ $\tilde{f}_{i,j}$ converge to a submetry, which is a canonical projection between Euclidean spaces
$$\tilde f_{i,\infty}:\left(\mathbb R^n,g_{i,\infty},0_i^n\right)\to \left(\mathbb R^m, h_{i,\infty}, 0_i^m\right), \qquad i=1,2.$$
At the same time, by \eqref{estimate-psi-blowup} and Arzela-Ascoli theorem,  $$\psi_j:\left(X_{1,j},\hat\epsilon_j^{-1}h_{1,j},f_{1,j}(x_j)\right)\to \left(X_{2,j},\hat\epsilon_j^{-1}h_{2,j},f_{2,j}(x_j)\right)$$ defined before Lemma \ref{lem-psi} converges to a $1$-Lipschitz map (resp. $L_0$-Lipschitz map) for \eqref{Closeness-of-vertical-distributions-1} (resp. for \eqref{Closeness-of-vertical-distributions-2})
$$\psi_{\infty}:(\mathbb R^m,h_{1,\infty},0_1^m)\to (\mathbb R^m,h_{2,\infty},0_2^m),$$ which by \eqref{estimate-psi-fib-blowup} satisfies
\begin{equation}\label{vertical-close-5}
\psi_\infty\circ \tilde f_{1,\infty}=\tilde f_{2,\infty}\circ I_{\infty}.
\end{equation}

We claim that

\noindent{\bf Claim}: For each $i=1,2$, after identifying $(\widetilde{U}_j, \tilde x_j, \hat\epsilon_j^{-1} \tilde g_{i,j})$ with $(\mathbb R^n, (\alpha_{i,j}^{-1})^*\hat\epsilon_j^{-1} \tilde g_{i,j})$ via harmonic coordinates chart $\alpha_{i,j}$ in the graph \eqref{blowing-up-epsilon}, the vertical distributions $\mathcal{V}_{\tilde{f}_{i,j}}(\tilde{x}_j)$ converge to $\mathcal{V}_{\tilde{f}_{i,\infty}}(0_1^n)$ in the sense that the dihedral angles between them goes to $0$ as $j\to\infty$.

By Lemma \ref{lem-limit-isomorphism}, $\tilde f_{2,\infty}\circ I_\infty$ is also a linear projection. Then \eqref{vertical-close-5} implies that the fibers of $\tilde f_{1,\infty}$ coincides with that of $\tilde f_{2,\infty}\circ I_\infty$. Hence, by Lemma \ref{lem-limit-isomorphism} and the Claim,  $\mathcal{V}_{\tilde{f}_{1,j}}(\tilde{x}_j))$ converges to the same limit as $\mathcal{V}_{\tilde{f}_{2,j}}(\tilde{x}_j)$, a contradiction to the assumption that $\mathcal{V}_{\tilde{f}_{1,j}}(\tilde{x}_j)$ and $\mathcal{V}_{\tilde{f}_{2,j}}(\tilde{x}_j)$ are definite away from each other by $\theta>0$.

The Claim easily follows from the $C^{1,\alpha}$-convergence \eqref{C1-convergence-1} and the fact by \eqref{estimate-2nd-fund-form-blowup} that each $\tilde f_{i,j}$-fiber is almost totally geodesic in $(\widetilde{U}_j, \tilde x_j, \hat\epsilon_j^{-1} \tilde g_{i,j})$.

Indeed, let $\sigma_{j}:[-1,1]\to \tilde{f}_{1,j}^{-1}(f_{1,j}(x_j))$ be any unit speed geodesic in $\tilde{f}_{1,j}^{-1}(f_{1,j}(x_j))$ equipped with the induced metric from $\hat\epsilon_{j}^{-1}\tilde{g}_{1,j}$ such that $\sigma_j(0)=\tilde{x}_j$ and $\sigma_j'(0)=v_j\in \mathcal{V}_{\tilde{f}_{1,j}}(\tilde{x}_j)\to v$ as $j\to \infty$. By \eqref{estimate-2nd-fund-form-blowup} and \eqref{C1-convergence-1}, it is almost a geodesic in the ambient space $(\widetilde{U}_j,\tilde g_{1,j})$ such that the 2nd derivative of its coordinate functions in the harmonic coordinates chart $\alpha_{1,j}$ converges to $0$ uniformly. Hence $\sigma_{j}$ $C^1$-converges to a line along $v$, which lies in the fiber of $\tilde f_{1,\infty}$ at $0_{1}^n$.  Similarly, $\mathcal{V}_{\tilde{f}_{2,j}}(\tilde{x}_j)$ converge to $\mathcal{V}_{\tilde{f}_{2,\infty}}(0_2^n)$ as $j\to\infty$.
\end{proof}

\section{Fibrations via smoothing methods}\label{proof-DWY-unique-structure}
 This section is devoted to establish the relation between fibrations arising from the original metric in Theorem \ref{thm-fibration} and those from nearby metrics with a sectional curvature bound by various smoothing methods.

As two concrete examples, we first consider the methods in \cite{DWY1996} and \cite{PWY1999} developed via Ricci flow \cite{Hamilton1982} and embedding $(M,g)$ to the Hilbert space $L^2(M,g)$ \cite{Abresch1988} respectively, and then consider the general case consisting of all nearby metrics.

The following lemma describes the convergence of nearby metrics and their limit spaces.

\begin{lemma}\label{lem-smoothed-limit}
	Let $(M,g)$ be a closed $n$-manifold
	with $|\Ric_M|\leq n-1$ and $(\delta,\rho)$-Reifenberg local covering geometry such that $d_{GH}((M,g),(X,h))\le \epsilon$, where $\delta=\delta(n)>0$ is a small constant and $(X,h)$ is a regular
	$(\delta,\rho)$-Reifenberg local covering Ricci limit space $(X,h)$, every $1$-ball on $(X,h)$ has volume $\ge v>0$, and $\diam(X,h)\le D$.
	
	Let $g(t)$ be a metric on $M$ for $t\in (0,T]$ with $T< \ln2$ such that
	\begin{equation}\label{ineq-nearby-metric}
	e^{-t}g\le g(t)\le e^tg,\; |\sec(g(t))|\le K(t), \text{ and } |\nabla^k \operatorname{Rm}(g(t))|_{g(t)}\le C(t,k) \; (k\ge 1)
	\end{equation}
	where $K(t)$ is a monotone decreasing function on $t$. Then
	$$d_{GH}((M,g(t)),(X,h(t))\le \varkappa(\epsilon\,|\,T,D,\rho,v,n),  \qquad \forall t\in (0,T]$$ where $h(t)$ is a smooth metric on $X$ whose sectional curvature $|\sec(h(t))|\le K_1(K(t),v)$ (does not depend on the high $k$-th order bound $C(t,k)$).
\end{lemma}
Note that, in general $K(t)\to \infty$ as $t\to 0$. A consequence of Lemma \ref{lem-smoothed-limit} is that, if the original limit space is regular, then any new limit space after a definite smoothing is also regular.

\begin{proof}
	Let us argue by contradiction. Assume that there is a sequence $(M_i,g_i)$ satisfying the conditions in Lemma \ref{lem-smoothed-limit} such that $d_{GH}((M_i,g_i),(X_i,h_i))\le \epsilon_i\to 0$, where $(X_i,h_i)$ are regular
	$(\delta,\rho)$-Reifenberg local covering Ricci limit spaces and every $1$-ball on $(X_i,h_i)$ has volume $\ge v>0$.
	
	By \cite[Theorem 0.4]{JKX2022Convergence}, the $C^{1,\alpha}$-radius of $(X_i,h_i)$ admits a uniform lower bound $r_0(n,v,\rho)>0$. Hence by passing to a subsequence, we may assume that $(X_i,h_i)$ $C^{1,\alpha}$-converges to $(X,h)$.
		
	Let $g_i(t_i)$ be smoothed metrics on $M_i$ satisfying \eqref{ineq-nearby-metric}. If $t_i\to 0$, then there is nothing to prove. Assume that $t_i\to t_0>0$. Suppose $(M_i,g_i(t_i))$ is $\epsilon_0$-definite away from $(X,h(t))$ for any Riemannian metric $h(t)$ with  $|\sec(h(t))|\le K_1(t_0,\rho,v)$.
	
	By passing to a subsequence, $(M_i,g_i(t_i))\overset{GH}{\longrightarrow} X_{t_0}$. Since $g_i(t_i)$ is $e^{t_i}$-bi-Lipschitz equivalent to $g$, $X_{t_0}$ is $e^{t_0}$-bi-Lipschitz equivalent to $(X,h)$. By the regularity of limit spaces under bounded sectional curvature by Fukaya \cite[Theorem 0.6]{Fukaya1988}, the tangent cone of such limit spaces is the quotient of an Euclidean space by an isometric torus action. Hence $X_{t_0}$ either is regular, or admits a tangent cone $\pi/2$-definite away from any Euclidean space (cf. \cite[Lemma 4.2]{JKX2022Convergence}). Because $e^{t_0}\le e^T<2$, $X_{t_0}$ is regular.
	
	Moreover, by the high $k$-th order bound on $\operatorname{Rm}(g(t))$, $X_{t_0}$ is the quotient of a smooth Riemannian manifold (\cite[Theorem 0.6]{Fukaya1988}). Hence $X_{t_0}$ is also a smooth Riemannian manifold $(X,h(t_0))$. By the regularity of limit spaces under bounded sectional curvature by Fukaya  \cite[Theorem 0.9]{Fukaya1988} (cf. \cite[Theorem 2.1]{JKX2022Convergence}), the sectional curvature bound of $(X,h(t_0))$ depends only on how much it is non-collapsing, i.e., $|\sec(h(t_0))|\le K_1(K(t_0),v)$. We meets a contradiction.

\end{proof}
\subsection{Smoothing via Ricci flow}\label{Ricci-flow}
Let us consider the following sequence in Problem \ref{problem-uniqueness-of-nilpotent-structure}
\begin{equation}\label{GH-convergence-to-X}
(M_i,g_i)\overset{GH}{\longrightarrow}(X,h),
\end{equation}
where $(X,h)$ is regular and every $1$-ball on $(X,h)$ has volume $\ge v>0$.

By \cite{DWY1996} (cf. \cite[Remark 1]{DWY1996}), we have
\begin{theorem}[\cite{DWY1996}, cf. {\cite{CRX2017}}]\label{thm-smoothing-ricci-flow}
	Given $n,\rho>0$, there exist constants $\delta=\delta(n),T(n,\rho)>0$ and $C(n,\rho)>0$ such that  if $(M,g)$ is a closed $n$-manifold
	with $|\Ric_M|\leq n-1$ and $(\delta,\rho)$-Reifenberg local covering geometry, then the Ricci flow equation
	\begin{equation}\label{ricci-flow-equation}
	\frac{\partial}{\partial t}g(t)=-2\Ric(g(t)),\qquad g(0)=g
	\end{equation}
	has a unique smooth solution $g(t)$ for $0<t\leq T(n,\rho)$ satisfying
	\begin{equation}\label{ineq-smoothing-sec}
	\left\{
	\begin{array}{llll}
	& \left|g(t)-g\right|_{g}\leq 4t;\\
	& \left|\operatorname{Rm}(g(t))\right|_{g(t)}\leq C(n,\rho)t^{-\frac{1}{2}};\\
	& \left|\nabla^{k}\operatorname{Rm}(g(t))\right|_{g(t)}\leq C(n,\rho,k,t);\\
	& \left|\Ric(g(t))\right|_{g(t)}\leq 2(n-1),
	\end{array}
	\right.
	\end{equation}
	where $\operatorname{Rm}(g(t))$ denotes the curvature tensor of $g(t)$, $\nabla^{k}\operatorname{Rm}(g(t))$ the $k^{th}$-convariant derivative of $\operatorname{Rm}(g(t))$ with respect to $g(t)$, whose norm is measured in $g(t)$.
\end{theorem}

Let $(M_i,g_i)$ and $(X,h)$ be as in \eqref{GH-convergence-to-X} and assume that $d_{GH}((M_i,g_i),(X,h))\leq \epsilon_i$ for each $i>0$, where $\epsilon_i\to 0$ as $i\to \infty$. Let $g_i(t), t\in (0,T(n,\rho)]$ be the solution of Ricci flow equation with initial metric $g_i(0)=g_i$ by Theorem \ref{thm-smoothing-ricci-flow}. By \eqref{ineq-smoothing-sec}, $g_i(t)$ is $e^{2t}$-bi-Lipschitz equivalent to $g_i$, and $|\Ric(g_i(t))|\le 2(n-1)$. Hence $(M_i,g_i(t))$ also admits $(\delta,\rho_1)$-Reifenberg local covering geometry, for some $\rho_1$ that depends on $\rho$ and $T(n,\rho)$.

By Lemma \ref{lem-smoothed-limit}, $(M_i,g_i(t))$ converges to a $(\delta,\rho_1)$-Reifenberg local covering Ricci limit space $X_t$, which is a smooth Riemaniann manifold $(X,h(t))$ and $e^{2t}$-bi-Lipschitz homeomorphic to $(X,h)$ with
\begin{equation}\label{sectional-curvature-bound-ht}
|\sec(X,h(t))|\le C(n,\rho,v)t^{-1/2}
\end{equation}
(Theorem 1.4, cf. \cite[Theorem 0.9]{Fukaya1988}). Since $h(t)$ is $e^{2t}$-bi-Lipschitz equivalence to $h$, every $1$-ball on $(X,h(t))$ has volume $\ge e^{-2nT(n,\rho)}v>0$.

By Theorem \ref{thm-fibration}, as $i$ large, there is an almost Riemannian submersion $f_{i,t}:(M_i,g_i(t))\to (X,h(t))$ for any $t\in [0,T(n,\rho)]$ satisfying \eqref{thm-fibration-1}-\eqref{thm-fibration-4} with $\epsilon=\epsilon_i\to 0$. By the stability Theorem \ref{thm-stability-nearby-metric}, all fibrations $f_{i,t}$ for $t\in [0,T(n,\rho)]$ are isomorphic to each other.

Let $\tilde g_i(t)=C(n,\rho,v)t^{-1/2}g(t)$ and $\tilde{h}(t)=C(n,\rho,v)t^{-1/2}h(t)$. Via \eqref{ineq-smoothing-sec} and \eqref{sectional-curvature-bound-ht}, one has,
\begin{equation}
\left|\sec\left(M_i,\tilde g_i(t)\right)\right|\leq 1, \quad \left|\sec\left(X,\tilde{h}(t)\right)\right|\leq 1.
\end{equation}
By the volume lower bound of $1$-balls in $(X,h(t))$, the injective radius of $(X,\tilde h(t))$  $\injrad(X,\tilde h(t))\geq r_1=r_1(n,\rho,v)>0$, for any $0< t\le T(n,\rho)$ (\cite[Theorem 4.3]{CGT1982}).
Without loss of generality, we may assume that $r_1\leq 1$.

Now let us apply the refined fibration theorem \cite[Theorem 2.6]{CFG1992} under bounded sectional curvature, i.e., there exists a small constant $\lambda(n)>0$ such that if $$d_{GH}(M_i,\tilde g_i(t)),(X,\tilde h(t))\le \epsilon \qquad\text{s.t. $\lambda^2=\epsilon\cdot r_1^{-1}\le   \lambda^2(n)$,}$$ then there is a fibration $\tilde f_{i,t}:\left(M_i,\tilde g_i(t)\right)\to \left(X,\tilde h(t)\right)$
with the following regularities:
\begin{enumerate}[(a)]
	\item\label{thm-rescaling-fibration-bounded-Ricci-conjugate-t-1}  $\tilde f_{i,t}$ is a $C(n)\lambda$-almost Riemannian submersion.
	\item\label{thm-rescaling-fibration-bounded-Ricci-conjugate-t-2}
	$\diam_{\tilde g_i(t)}(\tilde f_{i,t}^{-1}(p))\leq C(n)\epsilon, \forall p\in X.$		
	\item\label{thm-rescaling-fibration-bounded-Ricci-conjugate-t-3}
	$\left|\nabla^{2}\tilde f_{i,t}\right|_{\tilde g_i(t),\tilde h(t)}\leq C(n)r_1^{-1}$.
\end{enumerate}	
By Lemma \ref{lem-smoothed-limit},  $d_{GH}((M_i,g_i(t)),(X,h(t))\cdot r_1^{-1}\cdot C(n,\rho,v)^{1/2}t^{-1/4}\le \lambda^2(n)$ as $i$ sufficient large. Hence such $\tilde f_{i,t}:(M_i,\tilde g_i(t))\to (X,\tilde h(t))$ exists.

By Theorem \ref{thm-stability-nearby-metric}, $\tilde f_{i,t}$ is also isomorphic to $f_{i,t}$. Note that after rescaling back, $\tilde f_{i,t}:(M_i,g_i(t))\to (X,h(t))$ also satisfies \eqref{thm-stability-almost-Riem-submerison}-\eqref{thm-stability-second-fundamental-form-diameter}, particularly the rescaling invariant control on $\sndf_{\tilde f_{i,t}}$ and $A_{\tilde f_{i,t}}$.

\subsection{Smoothing via embedding method}
By embedding $(M,g)$ into the Hilbert space $L^2(M,g)$ by solutions of certain geometric PDE, Petersen-Wei-Ye in \cite{PWY1999} showed that for any given $0<\alpha<1$ and $Q>0$, if the weak harmonic $C^{0,\alpha}$-norm of a complete Riemannian $n$-manifold $(M,g)$ on scale $r>0$ is bounded above by $Q(r)$ (cf. section 2 in \cite{PWY1999}), where $Q(r)$ is non-decreasing in $r$ and $Q(r)\to 0$ as $r\to 0$, then $(M,g)$ admits a smoothed metric $g(t)$ for any $t>0$ such that
\begin{equation}\label{ineq-smoothing-sec-by-embedding}
\left\{
\begin{array}{llll}
& e^{-t}g\leq g(t)\leq e^{t}g;\\
& \left|\sec(g(t))\right|\leq K(t,n,\alpha,Q(r));\\
& \left|\nabla^{g(t)}-\nabla^{g}\right|_{g}\leq C(t,n,\alpha,Q(r));\\
&|\nabla^k \operatorname{Rm}(g(t))|_{g(t)}\le C(t,n,\alpha,Q(r),k) \; (k\ge 1)
\end{array}
\right.
\end{equation}
where $\nabla^g$ and $\nabla^{g(t)}$ are the Levi-Civita connections of $g$ and $g(t)$ respectively. The uniform control on the difference of Levi-Civita connections follows from the uniform bound on the 2nd fundamental form of the embedding (\cite[\S 5]{PWY1999}), and $K,C$ in \eqref{ineq-smoothing-sec-by-embedding} blow up as $t\to 0$.

Let $(M_i,g_i)$ and $(X,h)$ be as in \eqref{GH-convergence-to-X} again. By $(\delta,\rho)$-Reifenberg local covering geomtry, the weak harmonic $C^{1,\alpha}$-norm of $(M_i,g_i)$ on scale $r$ is bounded above by $Q(r,n,\rho,\alpha)$. Hence, by \cite{PWY1999} there is a family of smooth metrics $g_i(t),t>0$ on $M_i$ for each $i>0$ satisfying \eqref{ineq-smoothing-sec-by-embedding}.

For any fixed $t\in [0,T]$ with $T<\ln 2$, by Lemma \ref{lem-smoothed-limit}, $(M_i,g_i(t))\overset{GH}{\longrightarrow}(X,h(t))$.
Hence by the refined fibration theorem \cite[Theorem 2.6]{CFG1992} under bounded sectional curvature, there are fibrations $f_{i,t}:(M_i,g_i(t))\to (X,h(t))$ satisfying \eqref{thm-stability-almost-Riem-submerison}-\eqref{thm-stability-second-fundamental-form-diameter} for any $i\ge N(t)$. By Theorem \ref{thm-stability-lipequiv-metric}, all $f_{i,t}$ are isomorphic to the fibration $f:(M,g)\to (X,h)$ in Theorem \ref{thm-fibration}.

Note that, the existence of $f_{i,t}$ depends on the curvature bound $K(t,n,\alpha,Q)$ and the Gromov-Hausdorff distance between $(M,g_i(t))$ and $(X,h(t))$, which relies on the lower bound of $t$ and the diameter of $(X,h)$.

\subsection{Arbitrary smoothing}
Now, let us consider the general case, i.e. Theorem \ref{thm-unique-smoothed}.

Let $(M,g)$ be a closed Riemannian $n$-manifold of $|\Ric_{M}|\leq n-1$ with $(\delta,\rho)$-Reifenberg local covering geometry for some $\delta=\delta(n)$ and let $(X,h)$ be an $m$-regular $(\delta,\rho)$-Reifenberg local covering Ricci limit space where every $1$-ball's volume $\geq v>0$.

Suppose that $g(t)$ and $h(t)$, $t\in (0,1]$ are any two Riemannian metrics on $M$ and $X$ respectively (maybe got by different smoothing methods, e.g. Theorem \ref{thm-smoothing}) satisfying
\begin{equation}\label{curvature-estimates-gt}
|\sec(M,g(t))|\leq K(t),\quad |\sec(X,h(t))|\leq K(t),
\end{equation}
and
\begin{equation}\label{bi-lipschitz-equivalent-distance-g(t)}
\left\{
\begin{array}{llll}
& e^{-t}d_{g}(x,y)\leq d_{g(t)}(x,y)\leq e^{t}d_{g}(x,y),\quad  \forall x,y\in M,\\
& e^{-t}d_{h}(p,q)\leq d_{h(t)}(p,q)\leq e^{t}d_{h}(p,q),\quad  \forall p,q\in X.
\end{array}
\right.
\end{equation}

Assume that $d_{GH}((M,g), (X,h))\leq \epsilon$. We will determine the relationship between $\epsilon$, $t$ and $K(t)$ such that there exists a fibration $f_{t}:(M,g(t))\to (X,h(t))$ which is isomorphic to $f:(M,g)\to (X,h)$ provided by Theorem \ref{thm-fibration}. In particular, we will remove the dependence on the diameter of $(M,g)$ appearing in the above two special cases.

Note that, without a uniform upper bound on the diameter of $(M,g)$, after blowing up, $(M, K(t)g(t))$ cannot not be globally close to $(X,K(t)h(t))$ in the Gromov-Hausdorff distance. We will apply a crucial observation from \cite{HKRX2020} that, there exists a $(1, \varkappa(\epsilon,t\,|\,K(t)))$-GHA, i.e., a map $\alpha: (M, K(t)g(t))\to (X,K(t)h(t)),$ which is $\varkappa(\epsilon,t\,|\,K(t))$-onto and restricting to every unit ball $B_1(p,K(t)g(t))$, $\alpha$ is a $\varkappa(\epsilon,t\,|\,K(t))$-isometry, and $\diam_{K(t)g(t)}(\alpha^{-1}(x))\leq \varkappa(\epsilon,t\,|\, K(t))$ for any $x\in X$ (cf. \cite{HKRX2020}). Huang-Kong-Rong-Xu \cite{HKRX2020} observed that such local GHA is also enough for the construction in \cite{CFG1992} and \cite{Fukaya1986} of a fibration $f_t:(M,g(t))\to (X,h(t))$ as $\varkappa(\epsilon,t\,|\,K(t))$ is sufficiently small.

We first prove the first part of Theorem \ref{thm-unique-smoothed}, i.e. the existence of $f_t:(M,g(t))\to (X,h(t))$ with suitable regularities.
\begin{lemma}\label{lemma-regularity-DWY}
	Let $(M,g),(X,h)$ and $g(t),h(t)$ be as above. There exists $\epsilon_1(n,\rho,v)>0$ such that if $d_{GH}((M,g),(X,h))\le \epsilon$ and $t+K^{1/2}(t)\epsilon\le \epsilon_1(n,\rho,v)$, then there is a fibration $f_{t}:(M,g(t))\to (X,h(t))$ satisfying the following properties:
	\begin{enumerate}
		\numberwithin{enumi}{theorem}
		\item\label{lemma-fibration-smoothng-riemannian-submersion}  $f_{t}:\left(M, g(t)\right)\to \left(X, h(t)\right)$ is a $C(n)(t+K^{1/2}(t)\epsilon)^{1/2}$-almost Riemannian submersion;
		\item\label{lemma-fibration-smoothing-fiber-diameter}  for any $p\in X,$  $\diam_{g(t)}(f_{t}^{-1}(p))\leq C(n,\rho,v)\epsilon,$
		\item\label{lemma-fibration-smoothing-second-fundamental}
		$\left|\nabla^{2} f_{t}\right|_{g(t),h(t)}\leq C(n,\rho,v)K^{1/2}(t),$		
		\item\label{lemma-fibration-smoothing-infranil} every $f_{t}$-fiber is an infra-nilmanifold.
	\end{enumerate}
\end{lemma}

\begin{proof}[Proof of Lemma \ref{lemma-regularity-DWY}]
	Let $\tilde{g}(t)=K(t)g(t)$ and $\tilde{h}(t)=K(t)h(t)$. From \eqref{curvature-estimates-gt}, one has
	\begin{equation*}
	\left|\sec\left(M,\tilde{g}(t)\right)\right|\leq 1, \quad \left|\sec\left(X,\tilde{h}(t)\right)\right|\leq 1.
	\end{equation*}
By the same discussion as before in \S3.1, we have
$$	\injrad(X,\tilde h(t))\geq r_1=r_1(n,\rho,v)>0,\quad \text{ for any }  0< t\le 1.$$
Let $\alpha:(M,g)\to (X,h)$ be an $\epsilon$-GHA. By \eqref{bi-lipschitz-equivalent-distance-g(t)}, it is directly to check that $\alpha:(M, \tilde g(t))\to (X,\tilde h(t))$ is a $(1,8t+K(t)^{1/2}\epsilon)$-GHA. By center of mass technique, we get a continuous  $(1,C_1(n)(t+K(t)^{1/2}\epsilon))$-GHA, still denoted by  $\alpha$.
\par
Let $\lambda(n)$ be the small constant in the refined fibration theorem \cite[Theorem 2.6]{CFG1992} (see also \S3.1)
and $\epsilon_1(n,\rho,v)=\frac{\lambda^{2}(n)r_1}{C_1}\leq 1$. For $t+K^{1/2}(t)\epsilon\le \epsilon_1$, by \cite[Theorem 2.6]{CFG1992} and via regularizing $\alpha$ one obtains a $C_{2}(n)(t+K^{1/2}(t)\epsilon)^{1/2}$-almost Riemannian submersion $f_{t}:\left(M,\tilde{g}(t)\right)\to \left(X,\tilde{h}(t)\right)$ with regularities
\begin{enumerate}[(a)]
	\item\label{Lemma-rescaling-fibration-bounded-Ricci-conjugate-t-2}
	$\diam_{\tilde g(t)}(f_{t}^{-1}(p))\leq C(n,\rho,v)K(t)^{1/2}\epsilon, \forall p\in X,$		
	\item\label{Lemma-rescaling-fibration-bounded-Ricci-conjugate-t-3}
	$\left|\nabla^{2} f_{t}\right|_{\tilde g(t),\tilde h(t)}\leq C(n)r_1^{-1}$
\end{enumerate}	
and \eqref{lemma-fibration-smoothing-infranil}.
Indeed, the fact that $f_{t}:\left(M,\tilde{g}(t)\right)\to \left(X,\tilde{h}(t)\right)$ is a $C_2(n)(t+K^{1/2}(t)\epsilon)^{1/2}$-almost Riemannian submersion, the regularities \eqref{Lemma-rescaling-fibration-bounded-Ricci-conjugate-t-3} and \eqref{lemma-fibration-smoothing-infranil} are directly from \cite[Theorem 2.6]{CFG1992}.
The proof of \eqref{Lemma-rescaling-fibration-bounded-Ricci-conjugate-t-2} follows from the same argument of (2.6.1) in \cite[Theorem 2.6]{CFG1992}.

Note that the proof of (2.6.1) in \cite[Theorem 2.6]{CFG1992} depends on the underlying bundle map is an almost Riemannian submersion, regularity \eqref{Lemma-rescaling-fibration-bounded-Ricci-conjugate-t-3} and the existence of $\epsilon$-GHA $\alpha: (M,g)\to (X,h)$. Here the main difference is that $f_{t}$ is a $C_2(t+K^{1/2}(t)\epsilon)^{1/2}$-almost Riemannian submersion depending on $t$, while the fibration maps constructed in \cite[Theorem 2.6]{CFG1992} are $\varkappa(\epsilon\,|\,n)$-almost Riemannian submersions. However, such a difference will not change the order of intrinsic diameter of fibres of $f_{t}$. For the convenience of the readers, the verification of \eqref{Lemma-rescaling-fibration-bounded-Ricci-conjugate-t-2} is as follows.

Up to a rescaling, let us assume that the injectivity radius of $(X,\tilde h(t))$ equals to $1$.
Suppose for some $p\in (X,\tilde h(t))$, $\diam_{\tilde g(t)}(f_{t}^{-1}(p))=\mu \epsilon,$
where $d_{GH}((M,g),(X,h))\le \epsilon$.
By \eqref{Lemma-rescaling-fibration-bounded-Ricci-conjugate-t-3}, the extrinsic diameter of other fibers over $B_{1/2}(p,\tilde h(t))$ are more than $C_3(n,\rho,v)\mu \epsilon$.
The fact that $f_{t}$ is a $C_2(t+K^{1/2}(t)\epsilon)^{1/2}$-almost Riemannian submersion implies that at least $e^{-mC_2(t+K^{1/2}(t)\epsilon)^{1/2}}C_4(n,\rho,v)\mu\epsilon^{-m}$ many $\epsilon$-balls with respect to metric $\tilde g(t)$ are required to cover $f^{-1}(B_{1/2}(p,\tilde h(t))).$
However, by the existence of $\epsilon$-GHA $\alpha:(M,g)\to (X,h)$ and \eqref{bi-lipschitz-equivalent-distance-g(t)}, at most $C_5(n)K^{1/2}(t)\epsilon^{-m}$ such balls are required.
By the choice of $t$ and $\epsilon$,
$C_2(t+K^{1/2}(t)\epsilon)\leq \lambda^{2}(n).$
Hence $\mu\le C_6(n,\rho,v)K^{1/2}(t)$.
\par
Let $C(n,\rho,v)=\max\{C(n)r_1^{-1}, C_6\}$ and $C(n)=C_2(n)$. After rescaling back the metrics, Lemma \ref{lemma-regularity-DWY} holds.	
\end{proof}



We are ready to prove Theorem \ref{thm-unique-smoothed}.
\begin{proof}[Proof of Theorem \ref{thm-unique-smoothed}]
	~
	
Let $(M,g)$, $(X,h)$ and $g(t)$ be as in Theorem \ref{thm-unique-smoothed} and let's choose   $h(t)$  as in Theorem \ref{thm-smoothing}. Put $X_t=(X, h(t)).$ Via Lemma \ref{lemma-regularity-DWY}, there is $t_0(n,\rho,v)=\frac{1}{2}\epsilon_1(n,\rho,v)>0$ so that for each $0<t\leq t_0$, if $K(t)^{1/2}\cdot d_{GH}((M,g),(X,h))\leq t_0$, then there is a fibration $f_t:(M,g(t))\to (X,h(t))$ with regularities \eqref{lemma-fibration-smoothng-riemannian-submersion}-\eqref{lemma-fibration-smoothing-infranil}. Then by Theorem \ref{thm-stability-nearby-metric},
$f_{t}:(M,K(t)g(t))\to (X,K(t)h(t))$ and $f:(M,K(t)g)\to (X,K(t)h)$ is isomorphic, as long as both $d_{GH}((M,g),(X,h))\le \epsilon$ and $t_0$ are sufficiently small such that $C(n)(t+K^{1/2}(t)\epsilon)^{1/2}\leq \epsilon_0(n,\rho,v)$ and $t_0\le \eta_0(n,\rho,v)$ in Theorem \ref{thm-stability-nearby-metric}.
\end{proof}

\section{The existence and uniqueness of a canonical nilpotent structure}\label{unique-nilpotent-structure}
Let $(M,g)$ be a closed Riemannian $n$-manifold with $|\Ric_M|\leq n-1$ and $(\delta,\rho)$-Reifenberg local covering geometry for some $\delta=\delta(n)$. Let $(X,h)$ be a compact $m$-regular $(\delta,\rho)$-Reifenberg local covering Ricci limit space with the volume of every $1$-ball on $M$ is greater than $ v>0$. Assume that $d_{GH}((M,g),(X,h))\leq \epsilon$, sufficiently small. Via Theorem \ref{thm-fibration}, there exists a fibration $f:(M,g)\to (X,h)$ with the regularities \eqref{thm-fibration-1}-\eqref{thm-fibration-4}.
We will first construct a canonical nilpotent Killing structure on $(M,g)$, whose underlying fibration coincides with $(M,X,f)$.

\subsection{Construction of a canonical nilpotent Killing structure}\label{construction-canonical-nilpotent-structure}

 Let $g(t)$ with $0<t\le T(n,\rho)$ be the smoothed metrics on $M$ provided by Theorem \ref{thm-smoothing-ricci-flow} via Ricci flow. Note that $g(t)$ also admits $|\Ric(M,g(t))|\le 2(n-1)$ and $(2\delta,2\rho_1)$-Reifenberg local covering geometry for some constants $\delta=\delta(n)>0$ and $\rho_1=\rho_1(n,\rho)>0$, provided $0<t\leq 1/2\ln(1+\delta)$. If $\diam(M,g)\le D$, then by Lemma \ref{lem-smoothed-limit}, there is a smoothed metric $h(t)$ on $X$ such that $(X,h(t))$ is also a $(2\delta,2\rho_1)$-Reifenberg local covering Ricci limit space, and $d_{GH}((M,g(t)),(X,h(t)))\le \varkappa
(\epsilon\,|\,T,D,\rho,v,n)$. By Theorem \ref{thm-fibration}, there is a fibration $\tilde{f}_t:(M,g(t))\to (X,h(t))$ satisfying the regularities \eqref{thm-fibration-1}-\eqref{thm-fibration-4} with respect to $g(t)$ and $h(t)$.  Via Theorem \ref{thm-stability-nearby-metric}, the fibration $(M,X,\tilde{f}_t)$ is isomorphic to $(M,X,f)$.

Note that for any fixed $t$, $g(t)$ is $A(t)$-regular, i.e. $|\nabla^k\operatorname{Rm}(g(t))|_{g(t)}\le A(n,\rho,k,t)$ $(k\ge 1)$. Hence $\tilde{f}_t$ also admits a uniform control on its higher derivatives for any fixed $t$. Since up to a blowing up, $\tilde{f}_t:(M,g(t))\to (X,h(t))$ satisfies the requirement for the construction of a nilpotent Killing structure on fibrations in \cite{CFG1992}, the same procedure as \cite[\S3-\S4]{CFG1992} yields
a sheaf $\tilde{\mathfrak{n}}_t$ of nilpotent Lie algebras of
vector fields on $M$, an action of the associated sheaf $\tilde{\mathfrak{n}}_t$ of simply connected nilpotent Lie groups, and an $\tilde{\mathfrak{n}}_t$-invariant metric $\tilde g_t$ such that all elements in $\tilde{\mathfrak{n}}_t$ are Killing fields of $\tilde g_t$. We define $\tilde{\mathfrak{n}}_t$ to be a canonical nilpotent structure on $(M,g)$ whose underlying fibration coincides with $(M,X,f)$.

Furthermore, without a uniform upper bound of $(M,g)$'s diameter.
Since  the metrics $g(\text{resp. } h)$ and $g(t)(\text{resp. } h(t))$ are  $e^{2t}$-bi-lipschitz, by a triangle inequality, there is a $(1,8t+\epsilon)$-GHA from $(M,g(t))$ to $(X,h(t))$.
Hence for $8t+\epsilon\le \epsilon_0$, the constant in Theorem \ref{thm-fibration}, a fibration $f_t:(M,g(t))\to (X,h(t))$ can still be constructed by Theorem \ref{thm-fibration}, which satisfies the regularities \eqref{thm-fibration-1} for $\varkappa(\epsilon+t\,|\,n)$, \eqref{thm-fibration-3} and \eqref{thm-fibration-4}. Moreover by the proof of \eqref{lemma-fibration-smoothing-fiber-diameter} of Lemma \ref{lemma-regularity-DWY}, \eqref{thm-fibration-2} also holds for $f_t$.
By the same discussion as above, there is a nilpotent structure $\mathfrak{n}_t$ on $M$ associated to $f_t$. By the stability Theorem \ref{thm-canonical-nilstr-two-metric} in the next subsection, we will see that if $\diam(M,g)\le D$, then $\mathfrak n_t$ is isomorphic to $\tilde{\mathfrak{n}}_t$, and thus it also gives a canonical nilpotent structure on $(M,g)$.

To complete the proof of Theorem \ref{thm-existence-canonical-nilstr} and for the reader's convenience let us briefly recall how a nilpotent Killing structure $\mathfrak{n}_t$ and a nearby $\mathfrak{n}_t$-invariant metric $\bar g_t$ are constructed in \cite[\S3-\S4]{CFG1992}.

When the base space is regular, a nilpotent Killing structure $(\mathfrak n, \tilde g_t)$ is equivalent to an affine bundle $(M,X,f,\mathcal{N})$, which at the same time is a Riemannian submersion from $(M,\tilde g_t)$ to its quotient space $M/\mathfrak{n}=(X,h_t)$. Recall that a fiber bundle $(M,X,f)$ is called to be \emph{affine}, if its fiber $F_p$ is diffeomorphic to an infra-nilmanifold $\Gamma\setminus\mathcal{N}$, and its structure group is contained in the affine transformation group of $\Gamma\setminus\mathcal{N}$, where $\mathcal{N}$ is a simply connected nilpotent Lie group, $\Gamma$ is a discrete subgroup of affine transformation group of  $(\mathcal{N},\nabla^{can})$ with index $\#(\Gamma\cap\mathcal{N})\setminus\Gamma\leq k(n)$, and $\nabla^{can}$ is the canonical flat connection on the tangent bundle $T\mathcal{N}$, which is, by definition, the unique connection that makes all the left invariant vector fields parallel.

In order to see the fibration, $f:M\to X$,  arising for $g(t)$ from Theorem \ref{thm-fibration} is an affine bundle, it is equivalent to show that it is a parametrized bundle version of Gromov's almost flat manifold, such that the structure group preserves the canonical flat connection on each fiber. Recall that by \eqref{thm-fibration-4} any fiber admits a universal cover which is a nilpotent Lie group $\mathcal{N}$. For $\Lambda\subset \operatorname{Aff}(\mathcal N,\nabla^{\operatorname{can}})$ be a discrete subgroup of affine groups that preserves the $\nabla^{\operatorname{can}}$, then it induces a connection, denoted also by $\nabla^{\operatorname{can}}$, on $\Lambda\setminus \mathcal N$, which is diffeomorphic to a fiber.

By Ruh \cite{Ruh1982} for each point $z$ in a fiber, there is a flat orthogonal connection $\nabla^z$ such that it is conjugate to $\nabla^{\operatorname{can}}$ on $\Lambda\setminus \mathcal{N}$ by a gauge transformation. It was showed in \cite[\S4]{CFG1992} that by suitable averaging $\nabla^z$, there is a canonical gauge equivalent flat connection on each fiber associated to $g(t)$, which is isomorphic to $\nabla^{\operatorname{can}}$ on $\Lambda\setminus \mathcal N$. Hence $(M,X,f)$ is an affine bundle, and $\mathfrak{n}_t$ is the sheaf of nilpotent Lie algebra by vector fields descending from right invariant fields on $\Lambda\setminus \mathcal{N}$.

Let $h(t)$ be chosen as in Theorem \ref{thm-smoothing} such that the $C^{1,\alpha}$-harmonic radius of $(X,h(t))$ is greater than $r_0(n,\rho,v,\alpha,Q)>0$ for any $t\in (0,T(n,\rho)]$. For any $p\in X$, let $\phi:B_{r_0}(p,h(t))\times f_t^{-1}(p)\to f_t^{-1}(B_{r_0}(p,h(t)))$ be a local trivialization of the almost Riemannian submersion $f_t:(M,g(t))\to (X,h(t))$. Denote $U=f_t^{-1}(B_{r_0}(p,h(t)))$ and let $\pi:(\tilde{U},\tilde{g}(t),\tilde{x})\to (U,g(t),x)$ be the Riemannian universal cover with $\pi(\tilde{x})=x\in U$. By the affine structure defined as above, $\tilde{U}$ is canonically diffeomorphic to $B_{r_0}(p,h(t))\times \mathcal{N}$, where $\mathcal{N}$ is a simply connected nilpotent Lie group and the fiber $f_t^{-1}(p)$ is diffeomorphic to $\Lambda\setminus \mathcal N$, where $\Lambda$ is isomorphic to the fundamental group of $U$. Thus $\mathcal{N}$ acts on $\tilde{U}$ globally, which is the left translation on each $\tilde{f}_t=f_t\circ \pi$-fiber, and induces the infinitesimal action on $U$ by $\mathfrak{n}_t$.

Now $\bar{g}_t$ is defined as follows. By \cite{Gromov1978} and \cite{Ruh1982} $\#(\Lambda\bigcap\mathcal{N})\setminus\Lambda\leq k(n)$, up to a finite cover we assume that $\Lambda\subset\mathcal{N}$. Let $v$ be a tangent vector at $y\in \tilde{U}$. Let $h\in \mathcal{N}$ and let $hv$ denote the image of $v$ under the differential of $h$. Then the function $h\to \tilde{g}(t)(hv,hv)$ is constant on the left cosets of $\Lambda$. Since $\mathcal{N}$ is a nilpotent Lie group, it is unimodular. Therefore, the space $\Lambda\setminus \mathcal N$ inherits a canonical invariant measure $d\mu$, of total volume $1$. Define
\begin{equation}
\hat{g}_t(v,v)=\int_{\Lambda\setminus \mathcal N}\tilde{g}(t)(hv,hv)d\mu,
\end{equation}
which is invariant under $\mathcal{N}$. Rescaling the metric $g(t)$ by $K(t)=C(n,\rho,v)t^{-1/2}$ ($C(n,\rho,v)$ is the constant in Theorem \ref{thm-smoothing} satisfying $|\sec(M,K(t)g(t))|\leq 1$ and $|\sec(X,K(t)h(t))|\leq 1$) and using \cite[Proposition 4.9]{CFG1992} and the estimate of the intrinsic diameter of $\mathfrak{n}_t$'s orbits measured in $g(t)$ (see \eqref{thm-fibration-2}), one has
\begin{equation}\label{C0-closeness-of-gt-and-bar-gt}
|\hat{g}_t-\tilde{g}(t)|_{\tilde{g}(t)}\leq C(n)K(t)^{\frac{1}{2}}\epsilon l^{-1}
\end{equation}
and
\begin{equation}\label{C-infty-closeness-of-gt-and-bar-gt}
|\nabla^{i}(\hat{g}_t-\tilde{g}(t))|_{\tilde{g}(t)}\leq C(n,A(t),i)K(t)^{\frac{1+i}{2}}\epsilon l^{-(1+i)}, i=1,2,\dots,
\end{equation}
where $\nabla^{i}$ denotes the $i^{th}$-convariant derivative with respect to $\tilde{g}(t)$ and $0<l=l(n,\rho,v)\leq 1$ is a constant such that $\injrad(X,K(t)h(t))\geq l$ (see the section \ref{Ricci-flow}).
\par
Let $\bar{g}_t$ be the pushes down metric to $M$ of $\hat{g}_t$. Clearly, the construction of $\bar{g}_t$ is independent of the choice of $U$ and of the choice of base point used to define $\tilde{U}$. Hence, $\bar{g}_t$ is a well defined Riemannian metric on $M$ and $\mathfrak{n}_t$-invariant. Obviously, $\bar{g}_t$ and $g(t)$ also satisfy the estimates \eqref{C0-closeness-of-gt-and-bar-gt} and \eqref{C-infty-closeness-of-gt-and-bar-gt}.
\par
In the following, let us prove Theorem \ref{thm-existence-canonical-nilstr}.
\begin{proof}[Proof of Theorem \ref{thm-existence-canonical-nilstr}]
	~
	
Let $(M,g)$ and $(X,h)$ be as in Theorem \ref{thm-fibration} and let $g(t),t\in (0,T(n,\rho)]$ be the solution of Ricci flow equation with initial value $g$. By the discussion above, there are $0<t_0(n,\rho,v),\epsilon_0(n,\rho,v)\leq 1$ such that if $d_{GH}((M,g),(X,h))\leq\epsilon\leq \epsilon_0$, then there is a nilpotent Killing structure $(\mathfrak{n}_{t_0},\bar{g}_{t_0})$ on $(M,g)$ from $g(t_0)$, whose underlying fibration coincides with $(M,X,f)$ provided by Theorem \ref{thm-fibration}.
\par
 Moreover, if $\epsilon_0$ is taken sufficiently small relative to $t_0$ and $l^{-1}$ such that the right-hand side of \eqref{C0-closeness-of-gt-and-bar-gt} and \eqref{C-infty-closeness-of-gt-and-bar-gt} is small, then $\bar{g}_{t_0}$ and $g(t_0)$ are $C^{2}$-close. Note that $|\Ric(M,g(t_0))|\leq 2(n-1)$, we have $|\Ric(M,\bar{g}_{t_0})|\leq 3(n-1)$. Furthermore, $(M,\bar{g}_{t_0})$ has $(3\delta,\rho_1)$-Reifenberg local covering geometry because $(M,g(t_0))$ has $(2\delta,2\rho_1)$-Reifenberg local covering geometry.
\end{proof}



\subsection{Stability of canonical nilpotent Killing structure}
In this subsection, we will prove Theorem \ref{thm-uniqueness-canonical-nilstr}. Let $(M,g)$ and $(X,h)$ be as in Theorem \ref{thm-fibration} and assume that $d_{GH}((M,g),(X,h))\leq \epsilon$.
Let $g_{\eta}$ be any nearby metric on $M$ such that $e^{-\eta}g\le g_{\eta}\le e^{\eta}g$ and $\max\{\max|\sec(M,g_{\eta})|,1\}=K_\eta\ge 1$. Assume $g_{\eta}$ admits a nilpotent Killing structure $\mathfrak{n}_{\eta}$, such that it's quotient space is a smooth Riemannian manifold $(X,h_{\eta})$ satisfying $e^{-\eta}h\le h_{\eta}\le e^{\eta}h$, and the diameter of $n_{\eta}$'s orbits measured in $g_{\eta}$ are no more than  $K_\eta^{-1/2}\epsilon_1$(sufficiently small).
Then $(M,K_\eta g_{\eta})$ has sectional curvature absolutely bounded exactly by $1$ and is still collapsed.
\par
By \eqref{bi-lip-equiv}, the quotient space $(X,h_{\eta})$ of $(M,g_{\eta})$ by $\mathfrak{n}_{\eta}$-action is a smooth Riemannian manifold. Hence the projection map
$f_{\eta}:(M, K_\eta g_{\eta})\to (X, K_\eta h_{\eta})$ is a smooth Riemannian submersion and satisfies $\diam_{K_\eta g_{\eta}} (f_{\eta}^{-1}(p))\le \epsilon_1$ for any $p\in X$. Via the construction of a nilpotent Killing structure under bounded sectional curvature in \cite{CFG1992}, the second fundamental form of $f_{\eta}$ above a regular point of the quotient space $X_\eta=(X,h_\eta)$ satisfies $|\nabla^{2}f_{\eta}|_{K_\eta g_{\eta},K_\eta h_{\eta}}\leq C_1(n,v)$ automatically, where generally the constant $C_1(n,v)$ depends on the distance to the singular set of $X_\eta$ (is empty in our case) and the volume of $X_\eta$. Note that $\vol(B_1(p,h))\geq v>0$ for any $p\in X$ and $h_{\eta}$ is $e^\eta$-bi-Lipschitz equivalent to $h$ by the assumption, one has $\vol(B_1(p,h_{\eta}))\geq e^{-n\eta}v>\frac{1}{2}v$ if $0\leq \eta\leq 1/n$.
Hence, by using O'Neill formula \cite{O'Neill1966} there is $C_2(n,v)>0$ such that $|\sec(X,K_\eta h_{\eta})|\leq C_2(n,v)$. By the estimate on local injectivity radius in \cite[Theorem 4.3]{CGT1982}, there is $i_0(n,v)>0$ such that $\injrad(X,K_\eta h_{\eta})\geq i_0$.
\par

Let $(\mathfrak{n}_{t_0},\bar{g}_{t_0})$ be the canonical nilpotent Killing structure defined in \S\ref{construction-canonical-nilpotent-structure} above for the smoothed metric $g(t_0)$ provided by Ricci flow, whose underlying fibration is $f_{t_0}:(M,\bar g(t_0))\to (X,\bar h(t_0))$.
In order to show the underlying fibrations $f_{\eta}$ and $f_{t_0}$ of the two nilpotent Killing structures $\mathfrak{n}_\eta$ and $\mathfrak{n}_{t_0}$ are isomorphic, let us first observe that the diameter of $f_{t_0}$-fibers is at the same scale of $f_\eta$.

\begin{lemma}\label{lem-small-diameter}
	After the same rescaling on the metric $\bar g_{t_0}$, $f_{t_0}$'s every fiber satisfies $\diam_{K_\eta \bar g_{t_0}}(f_{t_0}^{-1}(p))\le C(n,\rho,v)(\epsilon_1+\eta)$.
\end{lemma}
\begin{proof}
	Since there is a Riemannian submersion $f_{\eta}:(M,K_\eta g_\eta)\to (X, K_\eta h_\eta)$, whose fiber has diameter $\le \epsilon_1$, by the bi-Lipschitz equivalence \eqref{bi-lip-equiv} it gives rise to a local $(1,2\epsilon_1+4\eta)$-GHA $ f_{\eta}:(M, K_\eta g)\to (X, K_\eta h)$. By repeating the proof of \eqref{lemma-fibration-smoothing-fiber-diameter} of Lemma \ref{lemma-regularity-DWY} for $f_{t_0}:(M, K_\eta g)\to (X, K_\eta h)$, we see that $\diam_{K_\eta g_{\eta}}(f_{t_0}^{-1}(p))\le C(n,\rho,v)(\epsilon_1+\eta)$ for any $p\in X$.
\end{proof}

By Lemma \ref{lem-small-diameter} and Theorem \ref{thm-stability-nearby-metric}, for  $0<C(n,\rho,v)\max\{\epsilon_1+\eta,\epsilon\} \leq \epsilon_0$, there is an isomorphism $(\Phi_1,\Psi)$ between the two fibrations such that $f_{t_0}\circ \Phi_1=\Psi\circ \tilde{f}_{\eta}$, where $\Psi:(X,h)\to (X,h_{\eta})$ is $4e^{\eta+t_0}$-bi-Lipschitz and $\Phi_1$ is $e^{\varkappa(\epsilon+\epsilon_1+t_0+\eta\,|\,n,\rho,v)}$-bi-Lipschitz.


Let $K=\max\{K_\eta, \max\{|\sec(M, g(t_0))|\}\}$. Then both $K g_\eta$ and $K g(t_0)$ admit $|\sec|\le 1$. Note that at least one of the followings holds:

(1) $\diam_{K g_{\eta}} f_\eta^{-1}(p)\le \epsilon_1$, for any $p\in X$;

(2) $\diam_{K g(t_0)} f_{t_0}^{-1}(p)\le \epsilon_2$, where $\epsilon_2$ is the smallest diameter of $f_{t_0}$ measured in $\max\{|\sec(M, g(t_0))|\}g(t_0)$.

Thus, by Lemma \ref{lem-small-diameter} again,
$$\max_{p\in X}\{\diam_{K g_{\eta}} f_\eta^{-1}(p), \diam_{K g(t_0)} f_{t_0}^{-1}(p)\}\le C(n,\rho,v)\max\{\epsilon_1,\epsilon_2\}.$$

Now the equivalence between $\mathfrak{n}_{t_0}$ and $\mathfrak{n}_{\eta}$ is reduced to the following stability theorem.

\begin{theorem}[cf. \cite{CFG1992}, \cite{JiangXu2019}]\label{thm-canonical-nilstr-two-metric}	
	Given $v>0$ and positive integers $m\leq n$, there are $\eta_0=\eta_0(n,v)>0$ and $\epsilon_0=\epsilon_0(n,v)>0$ such that for any $0\leq \eta\leq \eta_0$ and $0<\epsilon\leq \epsilon_0$, the following holds.
	\par
	Assume $(M,X_i,f_i,\mathcal{N}_i)$ $(i=1,2)$ are two affine bundle associated to nilpotent Killing structures $(\mathfrak{n}_i,g_i)$ such that $|\sec(M,g_i)|\leq 1$,
	$g_i$ are $e^{\eta}$-bi-Lipschitz equivalent to each other, and every $1$-ball on the quotient spaces $(X_i,h_i)$ by $\mathfrak{n}_i$ has volume no less than $v>0$. If $f_i:(M_i,g_i)\to (X_i,h_i)$ satisfies the following regularities:
\begin{enumerate}
		\numberwithin{enumi}{theorem}
\item \label{diameter-of-affine-fiber} for any $p_i\in X_i$, $\diam_{g_i}(f_i^{-1}(p_i))\leq \epsilon;$
\item \label{bound-second-fundamental-form-affine-fiber} the second fundamental form of $f_i$ satisfies $|\nabla^{2}f_i|_{g_i,h_i}\leq 1.$
\end{enumerate}
	Then there is an affine bundle isomorphism $(\Phi,\Psi)$ such that $\Psi\circ f_1=f_2\circ \Phi$ and $\Phi$ preserve the affine structure.
\end{theorem}

\par
The argument of Theorem \ref{thm-canonical-nilstr-two-metric} is similar as that of \cite[Theorem 3.1]{JiangXu2019} (cf. also \cite[Section 7]{CFG1992}). For completeness, we give a proof below. Before giving the proof, let us make some preparations.
\par
Since $\mathfrak{n}_i$ is a nilpotent Killing structure for $g_i$ and $h_i$ is the quotient metric of $g_i$ on $X_i$ by $\mathfrak{n}_i$-action, $f_i:(M,g_i)\to (X_i,h_i)$ is a Riemannian submersion.
By \eqref{bound-second-fundamental-form-affine-fiber} and O'Neill formula \cite{O'Neill1966}, one has $|\sec(X_i,h_i)|\leq 4$. Combine this with the assumption $\vol(B_1(p_i,h_i))\geq v>0$ for any $p_i\in X_i$ and the estimate on local injectivity radius in \cite[Theorem 4.3]{CGT1982}, we have $\injrad(X_i,h_i)\geq i_1(n,v)$ for some constant $i_1(n,v)>0$.
\par
Via the assumption in Theorem \ref{thm-canonical-nilstr-two-metric}, the discussion above and Theorem \ref{thm-stability-nearby-metric}, there are $\eta_0(n,v)>0$ and $\epsilon_0(n,v)>0$ such that if $0\leq \eta\leq \eta_0$ and $0<\epsilon\leq \epsilon_0$, there is a bundle isomorphism $(\Phi_1,\Psi)$ between $(M,X_i,f_i)$ such that $\Psi\circ f_1=f_2\circ \Phi_1$. However, $\Phi_1$ don't preserve the affine structure in general. Hence, we need to improve $\Phi_1$ to a diffeomorphism $\Phi:M\to M$ such that $\Psi\circ f_1=f_2\circ \Phi$ and $\Phi$ preserve the affine structure, which is equivalent to that for any $x\in M$ and the neighborhood $U=f_1^{-1}(B_{i_1}(f_1(x),h_1))$ of $x$,  the two actions induced by $\mathfrak{n}_i$ are conjugate by the lifting of $\Phi$ on the universal cover $\pi:(\tilde{U},\tilde x)\to (U,x)$ of $U$.
\par
Put $\tilde{f}_1=f_1\circ \pi$ and $\tilde{f}_2=f_2\circ\Phi_1\circ\pi$.
Let $(M,X_2,f_{2}^{*},\mathcal{N}_{2})$ denote the pullback affine bundle of $(M,X_2,f_2,\mathcal{N}_2)$ by diffeomorphism $\Phi_1$, where $f_2^*=f_2\circ\Phi_1$.
Let $\rho_1(\mathcal{N}_1)$(resp. $\rho_2(\mathcal{N}_2)$) be the free action of the simply connected nilpotent Lie group $\mathcal{N}_{1}$(resp. $\mathcal{N}_2$) on $\tilde{U}$ induced by $\mathfrak{n}_1$ (resp. $\mathfrak{n}_2$),
which are left translation on $\tilde{f}_{1}$-fibres (resp. $\tilde{f}_{2}$-fibres). And let $\Lambda$ denote the fundamental group of $f_{1}^{-1}(p)=(\Psi^{-1}\circ f_2\circ\Phi)^{-1}(p)$. By Malcev's rigidity theorem (see \cite{Rag79},\cite{BK1981} or \cite[Theorem 3.7]{CFG1992}), $\mathcal{N}_1$ and $\mathcal{N}_{2}$ can be identified to be the  same group $\mathcal{N}$ by the natural isomorphism between their lattice $\Lambda\cap \mathcal{N}_1$ and $\Lambda\cap \mathcal{N}_{2}$. Moreover, the two actions $\rho_1(\mathcal{N})$ and $\rho_2(\mathcal{N})$ of $\mathcal{N}$ coincide on $\Lambda$.
\par
Let $\tilde{g}_1$ (resp. $\tilde{g}_2$) be the pullback metric of $g_1$ (resp. $\Phi_1^{*}g_2$) to $\tilde{U}$ by $\pi$. Since $g_i,i=1,2$ are affine invariant with respect to $f_i$, the action $\rho_i(\mathcal{N})$ is isometric with respect to $\tilde{g}_i$. Moreover, by \cite[Proposition 4.6.3]{BK1981} (or \cite[Lemma 7.13]{CFG1992}),
\begin{equation}\label{injrad-estimate-of-tilde-g1}
\injrad_{\tilde{g}_i}(\tilde{y})\geq \min\{i_1/2, i_2(n)\}>0
\end{equation}
for any $\tilde{y}\in \tilde{U}$ with $d_{\tilde{g}_i}(\tilde{y}, \partial\tilde{U})>i_1/2,$ where $i_2(n)>0$ is a constant.
\par
\begin{lemma}\label{two-actions-C1-close}
The two actions $\rho_i(\mathcal{N}), i=1,2$ on $\tilde{U}$ are $\varkappa(\epsilon,\eta\,|\,n,\rho,v)$-$C^1$-close.
\end{lemma}
\begin{proof}
We argue it by contradiction. Assume that there exist $(M_j,g_{i,j})$ and $(X_{i,j},h_{i,j})$ be as in Theorem \ref{thm-canonical-nilstr-two-metric}, where $i=1,2$, $j=1,2,\dots$ and a sequence of affine bundles $(M_j,X_{i,j},f_{i,j},\mathcal{N}_{j})$
satisfying the regularities \eqref{thm-fibration-1}-\eqref{thm-fibration-4} (replacing $\epsilon$ by $\epsilon_j$, where $\epsilon_j\to 0$ as $j\to \infty$) such that $g_{i,j},i=1,2$ are $e^{\eta_j}$-bi-Lipschitz equivalent for each $j>0$, where $\eta_j\to 0$ as $j\to \infty$ and affine invariant with respect to $f_{i,j}$. However, there are $p_{1,j}\in X_{1,j}$, $U_j=f_{1,j}^{-1}(B_{i_1}(p_{1,j}))$ and $x_j\in U_j$ such that the two actions $\rho_{i,j}(\mathcal{N}_j)$ on the universal cover $\tilde{U}_j$ of $U_j$ are $\theta_0$ away from each other in the sense of $C^{1}$ at some point $\tilde{x}_j\in \pi_j^{-1}(x_j)$ for all $j>0$, where $\pi_j:\tilde{U}_j\to U_j$ is the universal cover and $\theta_0$ is a positive constant.
 Moreover, we can assume that the actions of $\rho_{i,j}(\mathcal{N}_j)$ coincide with on $\Lambda_j$, where $\Lambda_j$ is the fundamental group of $U_j$.
\par
By the discussion above and by passing to a subsequence we have the following equivalent convergence
  \begin{equation}\label{closeness-of-two-actions}
   (\tilde{U}_j, \tilde g_{i,j},\tilde x_j, \rho_{j}(\mathcal{N}_j))\stackrel{C^{1,\alpha}}{\longrightarrow} (\tilde U_{i,\infty}, \tilde g_{i,\infty},\tilde x_{i,\infty}, \rho_{i,\infty}(\mathcal{N}_{i,\infty})), \quad \text{as}\ j\to \infty,
   \end{equation}
  where $\tilde{g}_{i,j}$ be the pull-back metrics defined as above.
Since the diameter of both $f_{1,j}$-fibers and $f_{2,j}$-fibers go to $0$ as $j\to \infty$, the action of $\Lambda_j$ becomes
more and more dense.  Thus, $\rho_{i,\infty}(\mathcal{N}_{i,\infty})$ is also the limit action of $\Lambda_j$.
\par
Let $(\Phi_{1,j},\Psi_j)$ be the bundle isomorphisms between $(M_{j},X_{i,j},f_{i,j})$ as above.  By Theorem \ref{thm-stability-nearby-metric}, 
we have
\begin{equation*}
e^{-\varkappa(\epsilon_j,\eta_j\,|\,n,\rho,v)}\leq|d\Phi_{1,j}(w)|_{g_{2,j}}\leq e^{\varkappa(\epsilon_j,\eta_j\,|\,n,\rho,v)},
\end{equation*}
where $d\Phi_{1,j}$ denotes the tangent map of $\Phi_{1,j}$.
\par
 By the definition of $\tilde{g}_{i,j}$ the identity maps $I_{j}:(\tilde{U}_j,\tilde{g}_{1,j},\tilde{x}_j)\to (\tilde{U}_j,\tilde{g}_{2,j},\tilde{x}_j)$ are $e^{\varkappa(\epsilon_j,\eta_j\,|\,n,\rho,v)}$-bi-Lipschitz diffeomorphisms, and so converge to an isometry $I_{\infty}:(\tilde{U}_{1,\infty},\tilde{g}_{1,\infty},\tilde x_{1,\infty})\to (\tilde{U}_{2,\infty},\tilde{g}_{2,\infty},\tilde x_{2,\infty})$ as $j\to \infty$. Thus, the two limit actions $\rho_{i,\infty}(\mathcal{N}_{i,\infty}),i=1,2$ can be identified by the isometry $I_{\infty}$. Note that $\rho_{i,\infty}(\mathcal{N}_{i,\infty})$ are all the limit action of $\Lambda_j$ as $j\to \infty$, $\rho_{i,j}(\mathcal{N}_j)$ are $\varkappa(\epsilon_j,\eta_j\,|\,n,\rho,v)$-$C^{0}$-close. Since $g_{i,j}$ are affine invariant with respect to $f_{i,j}$, $\rho_{i,j}(\mathcal{N}_j)$ acts on $(\tilde{U},\tilde{g}_{i,j})$ isometrically. Hence, it follows from \eqref{closeness-of-two-actions} that $\rho_{i,j}(\mathcal{N}_j)$ are $\varkappa(\epsilon_j,\eta_j\,|\,n,\rho,v)$-$C^{1}$-close, a contradiction with the assumption.
\end{proof}


In the following, similar as in \cite[Section 7]{CFG1992} (cf. also \cite[Section 7]{JiangXu2019}), we will construct a diffeomorphism $\Phi_2:M\to M$ by using Lemma \ref{two-actions-C1-close} such that the lifting of $\Phi_2$ to the universal covering space $\tilde{U}$ of $U=f_1^{-1}(B_{i_1}(p,h_1))$ communicate with the two actions $\rho_i(\mathcal{N}),i=1,2$ of $\mathcal{N}$ for any $p\in X_1$. It shows that $\Phi_2$ is an affine bundle isomorphism between $(M,X_1,f_1,\mathcal{N}_1)$ and $(M,X_2,f_2^*=f_2\circ \Phi_1,\mathcal{N}_2)$.
Hence, $(\Phi=\Phi_1\circ\Phi_2,\Psi)$ is the desired affine bundle isomorphism between $(M,X_i,f_i,\mathcal{N}_i),i=1,2$.
\begin{proof}[Proof of Theorem \ref{thm-canonical-nilstr-two-metric}]
	~
	
 Let $\tilde{U}_1=\tilde{f}_1^{-1}(B_{i_1/2}(p,h_1))$ and let $\rho_i(\mathcal{N}),i=1,2$ are the isometric actions of simply connect nilpotent Lie group $\mathcal{N}$ on $\tilde{U}$ induced by $\tilde{f}_i$. Let $\Lambda$ be the fundamental group of $U=f_1^{-1}(B_{i_1}(p,h_1))$ and let $\rho_i(\Lambda)$ denote the actions of $\Lambda$ on $\tilde{U}$. Then $\rho_1(\lambda)=\rho_2(\lambda)$ for any $\lambda\in \Lambda$.  Hence, for any $[h]\in \Lambda\setminus \mathcal N$, $\rho_1([h]^{-1})\rho_2([h])$ is well defined and $\varkappa(\epsilon,\eta\,|\,n,\rho,v)$-$C^1$-close to identity by Lemma \ref{two-actions-C1-close}.  By the assumption of curvature condition of $(M,g_1)$ and \eqref{injrad-estimate-of-tilde-g1}, for any
   $\tilde y\in \tilde f_1^{-1}(B_{r_0/2}(p,h_1))$ we can define $\tilde\Phi_1(\tilde y)$ to be the center of mass of $[h]\to \rho_1([h]^{-1})\circ \rho_2([h]),  [h]\in \Lambda\setminus \mathcal N$, i.e., the critical value of
$$\tilde{z}\to \frac{1}{2}\int_{\Lambda\setminus \mathcal N}d_{\tilde{g}_1}^2(\tilde{z}, \rho_1([h]^{-1})\circ\rho_2([h])(\tilde y))d\mu,$$
where $d\mu$ is a measure on $\Lambda\setminus \mathcal N$ of total volume $1$.
 By \cite{GroveKarcher1973}, $\tilde \Phi_2:\tilde{U}_1\to \tilde{V}_1=\tilde{\Phi}_2(\tilde{U}_1)$ is a diffeomorphism such that
 \begin{equation*}
 \tilde{\Phi}_2\circ \rho_2(h)=\rho_1(h)\circ\tilde{\Phi}_2, \ \text{for any}\ h\in \mathcal{N}.
 \end{equation*}
 By construction, the quotient $\Phi_2:f_1^{-1}(B_{i_1/2}(p,h_1))\to f_1^{-1}(\tilde{f}_1(\tilde{V}_1))$ is affine equivalent. Moreover, for any $x\in B_{i_1/2}(p,h_1)$, the definition of $\Phi_2(x)$ does not depend on the choice of $p$. Hence, it can be extended to a globally defined diffeomorphism $\Phi_2:M\to M$, which is an affine bundle isomorphism between $(M,X_1,f_1,\mathcal{N}_1)$ and $(M,X_2,f_2^*,\mathcal{N}_2)$. Therefore, $(\Phi=\Phi_1\circ \Phi_2,\Psi)$ is an affine bundle isomorphism between $(M,X_1,f_1,\mathcal{N}_1)$ and $(M,X_2,f_2,\mathcal{N}_2)$ and $\Phi$ preserve the affine structure.
 \end{proof}
\begin{proof}[Proof of Theorem \ref{thm-uniqueness-canonical-nilstr}]
	~
	
It follows from Theorem \ref{thm-canonical-nilstr-two-metric} and the discussion above Theorem \ref{thm-canonical-nilstr-two-metric} immediately.
\end{proof}

\bibliographystyle{plain}
\bibliography{document}
\end{document}